\documentclass[a4paper,11pt,twoside]{amsart}

\usepackage{graphicx}
\usepackage{amsfonts}
\usepackage{amsmath}
\usepackage{amsthm}
\usepackage{amssymb}
\usepackage[all]{xy}
\usepackage{hyperref}
\usepackage{aliascnt}
\usepackage[mathscr]{euscript}

\newtheorem{thm}{Theorem}[section]

\newtheorem{thmInt}{Theorem}[section]

\newaliascnt{prop}{thm}
\newtheorem{prop}[prop]{Proposition}
\aliascntresetthe{prop}
\newaliascnt{lem}{thm}
\newtheorem{lem}[lem]{Lemma}
\aliascntresetthe{lem}

\newaliascnt{cor}{thm}
\newtheorem{cor}[cor]{Corollary}
\aliascntresetthe{cor}

\theoremstyle{definition}

\newaliascnt{definition}{thm}
\newtheorem{definition}[definition]{Definition}
\aliascntresetthe{definition}

\newaliascnt{remark}{thm}
\newtheorem{remark}[remark]{Remark}
\aliascntresetthe{remark}

\newaliascnt{ex}{thm}
\newtheorem{ex}[ex]{Example}
\aliascntresetthe{ex}

\newaliascnt{qn}{thm}
\newtheorem{qn}[qn]{Question}
\aliascntresetthe{qn}

\newaliascnt{stp}{thm}

\aliascntresetthe{stp}

\newaliascnt{setup}{thm}

\aliascntresetthe{setup}

\numberwithin{equation}{section}

\DeclareMathOperator{\im}{im} 
\DeclareMathOperator{\cok}{coker} 

\newcommand{\iso}{\cong}

\newcommand{\all}{\forall\,}

\newcommand{\st}{\,|\,} 
\newcommand{\comp}{\circ} 
\newcommand{\rest}[1]{|_{#1}} 

\newcommand{\id}{\mathrm{id}}
\newcommand{\Hom}{\mathrm{Hom}}
\newcommand{\Ext}{\mathrm{Ext}}

\newcommand{\mor}[1]{\xrightarrow{#1}} 
\newcommand{\mono}{\hookrightarrow} 
\newcommand{\epi}{\twoheadrightarrow} 

\newcommand{\kk}{\Bbbk} 
\newcommand{\NN}{\mathbb{N}}
\newcommand{\ZZ}{\mathbb{Z}}

\newcommand{\s}[1]{\mathcal{#1}} 
\newcommand{\so}{\s{O}} 
\newcommand{\sHom}{\s{H}om} 

\newcommand{\opp}{^{\circ}} 
\newcommand{\farg}{-} 
\newcommand{\ort}[1]{\langle#1\rangle} 
\newcommand{\sh}[2][1]{#2[#1]} 

\newcommand{\cat}[1]{{\mathscr{#1}}} 
\newcommand{\ca}{\cat{A}}
\newcommand{\cA}{\cat{A}}
\newcommand{\cb}{\cat{B}}

\newcommand{\cc}{\cat{C}}
\newcommand{\cC}{\cat{C}}
\newcommand{\cd}{\cat{D}}

\newcommand{\cP}{\cat{P}}

\newcommand{\cs}{\cat{S}}

\newcommand{\ct}{\cat{T}}
\newcommand{\cT}{\cat{T}}

\newcommand{\Proj}{\cP}

\newcommand{\scat}[1]{{\mathbf{#1}}} 
\newcommand{\Coh}{\scat{Coh}}
\newcommand{\Qcoh}{\scat{Qcoh}}
\newcommand{\dgCat}{\scat{dgCat}} 
\newcommand{\Hqe}{\scat{Hqe}} 
\newcommand{\D}{\scat{D}} 
\newcommand{\Da}{\D^?}

\newcommand{\Dm}{\D^-}
\newcommand{\Db}{\D^b}
\newcommand{\Dsb}{\D^{sb}}

\newcommand{\Dq}{\D_{\scat{qc}}}

\newcommand{\Dp}{\scat{Perf}} 
\newcommand{\K}{\scat{K}} 
\newcommand{\Ka}{\K^?}

\newcommand{\Km}{\K^-}
\newcommand{\Kb}{\K^b}
\newcommand{\C}{\scat{C}} 

\newcommand{\dgC}{\C_{\scat{dg}}} 

\newcommand{\Acya}{\Ka_{\mathrm{acy}}}

\newcommand{\V}{\scat{V}}
\newcommand{\Va}{\V^?}

\newcommand{\B}{\scat{B}}
\newcommand{\Ba}{\B^?}

\newcommand{\Perf}[1]{\mathrm{Perf}(#1)} 

\newcommand{\Mod}[1]{\mathrm{Mod}(#1)} 
\newcommand{\dgMod}[1]{\mathrm{dgMod}(#1)}

\newcommand{\dgHom}{\underline{\sHom}} 

\newcommand{\fun}[1]{\mathsf{#1}} 
\newcommand{\fD}{\fun{D}}

\newcommand{\fF}{\fun{F}}
\newcommand{\fG}{\fun{G}}
\newcommand{\fH}{\fun{H}}
\newcommand{\fI}{\fun{I}}

\newcommand{\fQ}{\fun{Q}}

\newcommand{\dgYon}[1][\cc]{\fun{Y}^{#1}_{\mathrm{dg}}} 

\newcommand{\ep}{\varepsilon}

\newcommand{\sm}[1][]{\mathrm{d}_{#1}} 
\newcommand{\Sq}{\scat{S}}
\newcommand{\Sqa}{\Sq_{\scat{a}}}
\newcommand{\Sqp}{\Sq_{\scat{p}}}
\newcommand{\Sqpa}{\Sqp^?}
\newcommand{\Sqpu}{\Sqp^+}
\newcommand{\Sqpm}{\Sqp^-}
\newcommand{\Sqpb}{\Sqp^b}
\newcommand{\Sqpsb}{\Sqp^{sb}}
\newcommand{\Sqpbfg}{\Sq_{\scat{p,fg}}^b}
\newcommand{\Sqpmfg}{\Sq_{\scat{p,fg}}^-}
\newcommand{\Sqpafg}{\Sq_{\scat{p,fg}}^?}
\newcommand{\Sqpc}{\Sqp^c}
\newcommand{\Sqd}[1][]{\Sq^d_{#1}}
\newcommand{\Sqpd}{\Sqp^d}
\newcommand{\Sqpad}{\Sqp^{?,d}}

\newcommand{\Sqpmd}{\Sqp^{-,d}}

\newcommand{\Sqpafgd}{\Sq_{\scat{p,fg}}^{?,d}}

\newcommand{\Sqpdc}{\Sqp^{\overline{c}}}
\newcommand{\Sqpadc}{\Sqp^{?,\overline{c}}}

\newcommand{\SQ}{\widehat{\Sq}}
\newcommand{\SQa}{\SQ_{\scat{a}}}
\newcommand{\SQp}{\SQ_{\scat{p}}}
\newcommand{\SQpa}{\SQp^?}
\newcommand{\SQpu}{\SQp^+}

\newcommand{\SQpb}{\SQp^b}
\newcommand{\SQpsb}{\SQp^{sb}}
\newcommand{\SQpbfg}{\SQ_{\scat{p,fg}}^b}
\newcommand{\SQpmfg}{\SQ_{\scat{p,fg}}^-}
\newcommand{\SQpafg}{\SQ_{\scat{p,fg}}^?}
\newcommand{\SQpc}{\SQp^c}
\newcommand{\SQd}[1][]{\SQ^d_{#1}}

\newcommand{\SQpad}{\SQp^{?,d}}

\newcommand{\SQpadc}{\SQp^{?,\overline{c}}}

\newcommand{\sq}[1]{#1^{\bullet}}
\newcommand{\quot}[1]{\overline{#1}}
\newcommand{\kep}{\kk[\ep]}
\newcommand{\Kac}{\K_{\scat{a}}}
\newcommand{\Khp}{\K_{\scat{p}}}
\newcommand{\Kco}{\K_{\scat{c}}}
\newcommand{\Kfr}{\K_{\scat{f}}}
\newcommand{\Kmi}{\K_{\scat{m}}}
\newcommand{\Kaf}{\K_{\scat{a,f}}}
\newcommand{\Kam}{\K_{\scat{a,m}}}
\newcommand{\Kpf}{\K_{\scat{p,f}}}
\newcommand{\Kpm}{\K_{\scat{p,m}}}

\newcommand{\Dmfg}{\Dm_\text{fg}}
\newcommand{\Dafg}{\Da_\text{fg}}
\newcommand{\Dcp}{\D^c}
\newcommand{\fMod}[1]{\mathrm{mod}(#1)} 
\newcommand{\compat}[1]{\mathrm{Com}(#1)} 
\newcommand{\inj}[1]{#1_{\mathrm{inj}}} 
\newcommand{\ninj}[1]{#1_{\mathrm{ni}}} 
\newcommand{\SHom}{\sq{\Hom}} 
\newcommand{\Homep}{\Hom^{\ep}} 
\newcommand{\Homph}{\Hom^{\mathrm{ph}}} 
\newcommand{\fin}{\fI}
\newcommand{\fpr}{\fQ}
\newcommand{\inc}{\iota}
\newcommand{\nat}{\tau} 
\newcommand{\nata}{\alpha}
\newcommand{\natb}{\beta}
\newcommand{\natc}{\gamma}
\newcommand{\natd}{\delta}
\newcommand{\nate}{\epsilon}
\newcommand{\cotr}[1]{#1_{\mathrm{ct}}} 

\begin{document}

	\author[A.~Canonaco]{Alberto Canonaco}
\address{A.C.: Dipartimento di Matematica ``F.\ Casorati''\\
        Universit{\`a} degli Studi di Pavia\\
        Via Ferrata 5\\
        27100 Pavia\\
        Italy}
	\email{alberto.canonaco@unipv.it}	
    
  \author[A.~Neeman]{Amnon Neeman}
    \address{A.N.: Dipartimento di Matematica ``F.~Enriques''\\Universit\`a degli Studi di Milano\\Via Cesare Saldini 50\\ 20133 Milano\\ Italy}
    \email{amnon.neeman@unimi.it}
    
    \author[P.~Stellari]{Paolo Stellari}
    \address{P.S.: Dipartimento di Matematica ``F.~Enriques''\\Universit\`a degli Studi di Milano\\Via Cesare Saldini 50\\ 20133 Milano\\ Italy}
    \email{paolo.stellari@unimi.it}
    \urladdr{\url{https://sites.unimi.it/stellari}}

 \title[Strong uniqueness of enhancements for the dual numbers]
       {Strong uniqueness of enhancements for the dual numbers: a case study}

 \thanks{A.~C.~is a member of GNSAGA (INdAM) and was partially supported by the research project PRIN 2022 ``Moduli spaces and special varieties''. A.~N.~was partly supported 
   by Australian Research Council Grants DP200102537 and DP210103397,
   and by ERC Advanced Grant 101095900-TriCatApp.
    	P.~S.~was partially supported by the ERC Consolidator Grant ERC-2017-CoG-771507-StabCondEn, by the research project PRIN 2022 ``Moduli spaces and special varieties'', and by the research project FARE 2018 HighCaSt (grant number R18YA3ESPJ)}



\begin{abstract}
We prove that the bounded and bounded below derived categories of (all) modules over the dual numbers have strongly unique (dg) enhancements. To this end we relate those categories to the category of sequences of vector spaces, which allows a complete classification of indecomposable objects. Along the way we also prove that all the derived categories of any hereditary category have strongly unique enhancements.
\end{abstract}

\maketitle

\setcounter{tocdepth}{1}
\tableofcontents

\section*{Introduction}

The fruitful interplay between algebraic geometry and homological algebra is a key feature of the recent history of both disciplines. In the old days, meaning back in the last decades of the 20-th century, the focus was on the theory of derived and triangulated categories naturally associated to schemes, with an emphasis on the bounded derived categories of coherent sheaves and on their semiorthogonal decompositions. As the relation between these two disciplines grew tighter and tighter, another key feature appeared to have high relevance: the fact that most triangulated categories admit so-called enhancements, meaning higher categorical models.

The passage to this new viewpoint provided the elegant and effective formalism of derived algebraic geometry and of the (more recent) infinity-category version. Both theories eliminate some of the well-known pathologies of triangulated categories. The new theories has seen beautiful applications and led to manifold achievements, which we will mostly not recount here. The reader is invited to study (for example) the vast literature on Fourier--Mukai functors, which are much clearer and more transparent in the enhanced setting.

Of course, this comes at a price: if nothing else, one has to ask oneself which constructions, if any, depend on the choice of higher categorical structure. 

Let us accept that enhancing triangulated categories might be relevant for geometric applications; it is not a priori clear which of the higher categorical enhancements is most suitable or efficient. There are at least three major approaches to higher categories. The first, inspired by algebraic geometry and representation theory, uses (pretriangulated) \emph{differential graded (dg) categories}. The second, coming from symplectic geometry/topology, uses the notion of (pretriangulated) $A_\infty$ \emph{categories.} The last one, inspired by algebraic topology, deals with (stable) \emph{$\infty$ categories}. A considerable and recent body of work has been devoted to the comparison between these three viewpoints and, luckily for us, it turns out that they are all equivalent under the reasonable assumption that all categories are linear over some commutative ring. The fact that the first and the second type of models are essentially the same is the content of \cite{COS1,COS2} while the comparison between the first and the third type of higher categories is the content of \cite{Coh,Doni}.

Given this, we may choose whichever of the three equivalent approaches we prefer. In this paper we will consider only dg enhancements. This leaves us with the second foundational question: do our constructions depend on the choice of dg enhancement? After all it is possible for a triangulated category to have different dg enhancements. Of course, each of them may be useful to deal with specific geometric problems, but nonetheless one would like to be able to compare different enhancements and, at best, prove that they are equivalent in a useful sense.

There turn out to be at least two different notions of uniqueness for enhancements, both introduced in the seminal paper \cite{LO} by Lunts and Orlov. More precisely, an enhancement for a triangulated category $\cT$ is a pair $(\cC,\fF)$ where $\cC$ is a pretriangulated dg category and $\fF\colon H^0(\cC)\to\cT$ is an exact equivalence. Then $\cT$ has a \emph{(strongly) unique enhancement} if, given two of them $(\cC_1,\fF_1)$ and $(\cC_2,\fF_2)$, there exists an isomorphism $f\colon\cc_1\to\cC_2$ in $\Hqe$ (such that $H^0(f)\iso\fF_2\comp\fF_1^{-1}$). Here $\Hqe$ denotes the localization of the category of dg categories with respect to the class of quasi-equivalences.

For uniqueness, the main results were obtained in a series of papers, starting with \cite{LO}, which we discuss in \autoref{subsec:enhuniq}. The current most general result was obtained in \cite{CNS}, and is reported as \autoref{thm:main2}: the derived category of any abelian category and many triangulated categories of geometric interest have unique enhancements.

Coming to strong uniqueness, it is easy to see that, once a triangulated category $\cT$ is known to have a unique enhancement, the enhancement is also strongly unique if and only if every exact autoequivalence $\fG$ of $\cT$ admits a dg lift. This means that, given an arbitrary enhancement $(\cC,\fF)$ of $\cT$, there exists an automorphism $f$ of $\cC$ in $\Hqe$ such that $\fG\iso\fF\comp H^0(f)\comp\fF^{-1}$. Thus, assuming uniqueness, proving strong uniqueness is a particular case of the more general problem of finding lifts of exact functors between triangulated categories. In good geometric cases, thanks to an important result by To\"en~\cite{To}, it is known that an exact functor has a lift if and only if it is of Fourier-Mukai type. Now, a celebrated theorem by Orlov~\cite{Or} states that every fully faithful exact functor $\Db(X)\to\Db(Y)$ between the bounded derived categories of coherent sheaves is of Fourier-Mukai type when $X$ and $Y$ are smooth and projective schemes over a field. Later this result was generalized in several ways---see, for example, \cite{CS,LO,CS1,Ol}). This provides several geometrically valuable instances of triangulated categories with strongly unique enhancement. However, in view of \autoref{thm:main2}, the current status is that strong uniqueness is known to hold for a much narrower class of categories than uniqueness. One reason for this is that a key ingredient, in all the known proofs of strong uniqueness mentioned above, is the presence of a so-called ample set. We do not need to define this here but, as the name suggests, an ample set is related to the possibility to pick objects that behave like ample line bundles and provide nice resolutions. Of course, even in the geometric case, it is easy to imagine situations where this cannot be achieved: when the scheme is not projective/proper or when we deal with larger categories such as the unbounded derived category $\Dq(X)$ of complexes of $\so_X$-modules with quasi-coherent cohomology.

What makes the strong uniqueness problem intriguing is the scarcity of counterexamples. Finding examples of triangulated categories with unique but not strongly unique enhancement seems a hard challenge. Indeed, even searching for exact functors---not necessarily equivalences---without a lift is a difficult task. Consider, for instance, that an influential conjecture from \cite{BLL} predicted that every exact functor $\Db(X)\to\Db(Y)$, with $X$ and $Y$ smooth and projective schemes over a field, should admit a lift. It took several years to disprove the conjecture (see \cite{RVN,V}). And coming back to our problem: so far the only known examples of unique but not strongly unique enhancement are those provided in \cite{JM}. Although interesting, they are in a sense `pathological': as ordinary functors, the autoequivalences without a lift are simply the identity, and what makes them non-trivial is the isomorphism of commutation with the shift, which is not the identity. And it is a wide-open problem if a counterexample can be found with a geometric triangulated category, as in \autoref{thm:main2}.

\subsection*{The result}

Back to the positive side of the story: a little is known, there have been
a few special cases in which strong uniqueness was proved without the use of ample sets. In each of these the proof is based on exploiting the particularly simple structure of the triangulated category under consideration; see \cite{AM,CY,Lo}. This paper goes in the same direction, proving strong uniqueness in some new cases, thanks to the fact that we can obtain a sufficiently nice description of every object in the triangulated category. An interesting feature of our results is that, for the first time to our knowledge, they apply not only to bounded derived categories. We will prove the following:

\begin{thmInt}\label{mainthm}
The following triangulated categories have a strongly unique enhancement:
\begin{itemize}
\item[{\rm (a)}] $\Da(\cA)$, where $?=b,+,-,\emptyset$ and $\cA$ is a hereditary abelian category;
\item[{\rm (b)}] $\Da(\Mod{\kep})$, for $?=sb,b,+$, and $\Dm(\fMod{\kep})$, where $\Mod{\kep}$ is the category of modules over the ring of dual numbers over a field $\kk$ and $\fMod{\kep}$ is its full subcategory of finitely generated modules.
\end{itemize}
\end{thmInt}

Some comments are in order here. As for item (a), recall that an abelian category is \emph{hereditary} if $\Ext^2(A,B)=0$, for all $A,B\in\cA$. We should also remind the reader that the case $?=b$ was already proved in \cite{Lo} (improving a result in \cite{CY}). Our proof, which works uniformly for every choice of $?$, is an easy application of techniques developed in \cite{CNS}.

Moving to item (b), we first point out that the meaning of the unusual superscript $sb$ is explained at the beginning of \autoref{subsec:derived}. Then we should note that the strong uniqueness of the enhancements, for the `small' triangulated categories $\Db(\fMod{\kep})$ and $\Perf{\kep}$ (the category of perfect complexes over $\kep$), was proved in \cite{AM,CY} (see also \cite{KYZ}). All these papers tackle the problem by classifying indecomposable objects and proving that every object is completely decomposable (meaning a direct sum of indecomposables). As the same approach is more problematic for the `big' derived categories $\Da(\Mod{\kep})$ we are interested in, we need to develop tools allowing us to deal efficiently with the latter categories. The crucial observation is that $\D(\Mod{\kep})$ is closely related to a subcategory of the category $\Sq=\Sq(\Mod{\kk})$, of \emph{sequences} of $\kk$-vector spaces. To classify the indecomposable objects, and to check which objects are completely decomposable, one can replace $\Da(\Mod{\kep})$ or $\Dm(\fMod{\kep})$ with the corresponding subcategory of $\Sq$. It turns out that everything works as expected when $?=sb,b,+$. However, due to the presence of non completely decomposable objects, understanding if strong uniqueness of enhancement holds for $\Da(\Mod{\kep})$ when $?=-,\emptyset$ is left as a stimulating open problem of linear algebra.

\subsection*{Structure of the paper}

As we mentioned above, the (small) derived categories,
of modules over the dual numbers, have been dealt with by means of a careful analysis of the indecomposable objects, and of the morphisms between them. Moving from $\Db(\fMod{\kep})$ to $\Da(\Mod{\kep})$ increases the complexity of the problem, and to approach this we develop the theory of categories of sequences---first in general, then in the special case of vector spaces. This is the content of \autoref{sequences}. Note that the indecomposable objects are classified in \autoref{subsec:indec}.

This material seems general, and might appear unrelated to the problem at hand. In \autoref{dercatdual} we explain the connection with the study of the homotopy and derived categories of modules over the dual numbers. With this done, the stage is set for the proof of \autoref{mainthm} (b), which is contained in \autoref{stronguniq}.

Of course, the proof requires also some machinery from the theory of dg categories, which is briefly recalled in \autoref{sec:enh}. In particular, the precise definitions of uniqueness and strong uniqueness for enhancements is discussed, together with \autoref{thm:main2} from \cite{CNS}. The proof of \autoref{mainthm} (a) is contained in \autoref{subsec:stronguniqher}. Beyond using some techniques developed in \cite{CNS}, it is based on the classical fact that every object in the derived category of a hereditary category splits as the coproduct of (the shifts of) its cohomologies. Thus we are back to our underlying idea that studying strong uniqueness of enhancements is easier for triangulated categories whose objects can be assembled out of simple building blocks.

In this paper is we are unable to settle the strong uniqueness of enhancements of $\Da(\Mod{\kep})$, for $?=\emptyset,-$. In \autoref{open} we give some partial results and list a few open problems.

\section{The category of sequences}\label{sequences}

In this section $\kk$ is an arbitrary commutative ring and $\ca$ is a $\kk$-linear category. In \autoref{subsec:defseq} we define two categories of sequences in $\ca$, and describe in detail the morphisms inside each of them. In \autoref{subsec:seqab} we assume that $\ca$ is abelian, and in \autoref{subsec:vectsp} and \autoref{subsec:indec} we further specialize to the case where $\kk$ is a field and $\ca=\Mod{\kk}$.

\subsection{Definitions}\label{subsec:defseq}

Let $\Sq(\ca)$ be the $\kk$-linear category of sequences in $\ca$. Explicitly, an object $\sq{X}$ of $\Sq(\ca)$ is given by objects $X^i\in\ca$ and morphisms $\sm^i=\sm[\sq{X}]^i\colon X^i\to X^{i+1}$ of $\ca$, for every $i\in\ZZ$. A morphism $\sq{f}\colon\sq{X}\to\sq{Y}$ in $\Sq(\ca)$ is given by morphisms $f^i\colon X^i\to Y^i$ of $\ca$ such that $\sm[\sq{Y}]^i\comp f^i=f^{i+1}\comp\sm[\sq{X}]^i$, for every $i\in\ZZ$. 

\begin{remark}\label{rmK:complseq}
In the above definition, we do not require $\sm^{i+1}\comp\sm^i=0$. But since this might happen, we recover the usual category of complexes $\C(\ca)$ as a full subcategory of $\Sq(\ca)$.
\end{remark}

Note that the usual shift functors $\sh[n]{}$ (for $n\in\ZZ$) extend from $\C(\ca)$ to $\Sq(\ca)$ with the same definition: $\sh[n]{\sq{X}}^i:=X^{n+i}$ and $\sm[{\sh[n]{\sq{X}}}]^i:=(-1)^n\sm[\sq{X}]^{n+i}$ on objects, whereas $\sh[n]{\sq{f}}^i:=f^{n+i}$ on morphisms.

\begin{remark}\label{Homsq}
As in the case of complexes,
\[
\Hom_{\Sq(\ca)}(\sq{X},\sq{Y})=\ker(\sm[\SHom_{\Sq(\ca)}(\sq{X},\sq{Y})]^0),
\]
where $\SHom_{\Sq(\ca)}(\sq{X},\sq{Y})\in\Sq(\Mod{\kk})$ is defined by
\begin{gather*}
\Hom^n_{\Sq(\ca)}(\sq{X},\sq{Y}):=\prod_{i\in\ZZ}\Hom_\ca(X^i,Y^{n+i}), \\
\sm[\SHom_{\Sq(\ca)}(\sq{X},\sq{Y})]^n(\sq{f}):=(\sm[\sq{Y}]^{n+i}\comp f^i-(-1)^nf^{i+1}\comp\sm[\sq{X}]^i)_{i\in\ZZ}
\end{gather*}
for every $n\in\ZZ$. The graded Leibniz rule holds: if $\sq{f}\in\Hom^m_{\Sq(\ca)}(\sq{X},\sq{Y})$ and $\sq{g}\in\Hom^n_{\Sq(\ca)}(\sq{Y},\sq{Z})$, then $\sq{g}\comp\sq{f}\in\Hom^{m+n}_{\Sq(\ca)}(\sq{X},\sq{Z})$ and
\[
\sm^{m+n}(\sq{g}\comp\sq{f})=\sm^n(\sq{g})\comp\sq{f}+(-1)^n\sq{g}\comp\sm^m(\sq{f}).
\]
It follows that, if $\sq{f}\in\ker(\sm^m)$ and $\sq{g}\in\im(\sm^n)$ or $\sq{f}\in\im(\sm^m)$ and $\sq{g}\in\ker(\sm^n)$, then $\sq{g}\comp\sq{f}\in\im(\sm^{m+n})$.
\end{remark}

For every $\sq{X},\sq{Y}\in\Sq(\ca)$ we set
\[
\Homep_{\Sq(\ca)}(\sq{X},\sq{Y}):=\cok(\sm[\SHom_{\Sq(\ca)}(\sq{X},\sq{Y})]^{-1}).
\]
Then we can consider a ($\kk$-linear) category $\SQ(\ca)$ (whose usefulness will be clear later) with the same objects as $\Sq(\ca)$ and whose morphisms are defined by
\[
\Hom_{\SQ(\ca)}(\sq{X},\sq{Y}):=\Hom_{\Sq(\ca)}(\sq{X},\sq{Y})\oplus\Homep_{\Sq(\ca)}(\sq{X},\sq{Y}).
\]
We can write a morphism $f\in\Hom_{\SQ(\ca)}(\sq{X},\sq{Y})$ as $f=\sq{f}_1+[\sq{f}_\ep]$, where $\sq{f}_1\in\Hom_{\Sq(\ca)}(\sq{X},\sq{Y})$ and $[\sq{f}_\ep]$ is the class in $\Homep_{\Sq(\ca)}(\sq{X},\sq{Y})$ of $\sq{f}_\ep\in\Hom^0_{\Sq(\ca)}(\sq{X},\sq{Y})$. Sometimes we will say that $f$ is of type $1$ (respectively, of type $\ep$) if $f=\sq{f}_1$ (respectively, $f=[\sq{f}_\ep]$).

Similarly, if $g=\sq{g}_1+[\sq{g}_\ep]\in\Hom_{\SQ(\ca)}(\sq{Y},\sq{Z})$, then the composition (which is well defined thanks to \autoref{Homsq}) is given by
\[
g\comp f:=\sq{g}_1\comp\sq{f}_1+[\sq{g}_1\comp\sq{f}_\ep+\sq{g}_\ep\comp\sq{f}_1].
\]
Note that the shift functors extend in an obvious way from $\Sq$ to $\SQ$.

\begin{remark}\label{essbij}
There are natural ($\kk$-linear) functors $\fin\colon\Sq(\ca)\to\SQ(\ca)$ and $\fpr\colon\SQ(\ca)\to\Sq(\ca)$, which are the identity on objects and which on morphisms are given, respectively, by the inclusion and the projection onto the first summand. It is clear that $\fpr$ is full and essentially surjective. Moreover, it is very easy to see that $\fpr$ reflects isomorphisms. As a full functor that reflects isomorphisms is essentially injective, $\fpr$ is essentially bijective.
\end{remark}

\subsection{The case of abelian categories}\label{subsec:seqab}

In the special case where $\ca$ is an abelian category we can first prove the following.

\begin{lem}\label{lem:abGroth}
If $\ca$ is an abelian (resp.\ a Grothendieck abelian) category, then $\Sq(\cA)$ is an abelian (resp.\ a Grothendieck abelian) category.
\end{lem}

\begin{proof}
This is an easy consequence of the fact that $\Sq(\ca)$ can be identified with the category of functors from $\ZZ$ (regarded as an ordered set, hence as a category) to the (Grothendieck) abelian category $\ca$.
\end{proof}

Recall that an abelian category $\cb$ is a \emph{Grothendieck category} if
\begin{itemize}
\item It has small coproducts;
\item Filtered colimits of exact sequences are exact in $\cb$;
\item It has a set of generators, meaning a set $\cs$ of objects in $\cb$ such that, for any $C\in\cb$ there is an epimorphism $S\epi C$, where $S$ is a coproduct of objects in $\cs$.
\end{itemize} 

We can also prove the following rather technical but useful result.

\begin{prop}\label{Ext1}
  If $\ca$ is an abelian category, then for every $\sq{X},\sq{Y}\in\Sq(\ca)$ there is an injective $\kk$-linear map, natural in the two arguments
  $\sq{X},\sq{Y}\in\Sq(\ca)$,
\[
\Homep_{\Sq(\ca)}(\sq{X},\sq{Y})\to\Ext_{\Sq(\ca)}^1(\sq{X},\sh[-1]{\sq{Y}}).
\]
Moreover, this map is an isomorphism if $\Ext_\ca^1(X^i,Y^{i-1})=0$ for every $i\in\ZZ$.
\end{prop}

\begin{proof}
Given $\sq{f}\in\Hom_{\Sq(\ca)}^0(\sq{X},\sq{Y})$, consider $\sq{C}\in\Sq(\ca)$ defined by\footnote{This is just the obvious generalization of the notion of mapping cone (up to shift) from complexes to sequences.} $C^i:=X^i\oplus Y^{i-1}$ and
\[
\sm[\sq{C}]^i=
\begin{pmatrix}
\sm[\sq{X}]^i & 0 \\
-f^i & -\sm[\sq{Y}]^{i-1 }
\end{pmatrix}
\colon C^i=X^i\oplus Y^{i-1}\to C^{i+1}=X^{i+1}\oplus Y^i.
\]
Clearly there is a short exact sequence in $\Sq(\ca)$
\[
E_{\sq{f}}\colon0\to\sh[-1]{\sq{Y}}\to\sq{C}\to\sq{X}\to0
\]
where the maps are the natural inclusion and projection. Thus we obtain a map
\[
\Hom_{\Sq(\ca)}^0(\sq{X},\sq{Y})\to\Ext_{\Sq(\ca)}^1(\sq{X},\sh[-1]{\sq{Y}}),
\]
which sends $\sq{f}$ to the isomorphism class of $E_{\sq{f}}$.

We leave to the reader the check that the above map is natural and $\kk$-linear.

In order to deduce the first statement it is then enough to show that $E_{\sq{f}}$ splits if and only if $\sq{f}\in\im(\sm[\SHom_{\Sq(\ca)}(\sq{X},\sq{Y})]^{-1})$. Indeed, $E_{\sq{f}}$ splits if and only if there exists a morphism (necessarily an isomorphism) $\sq{g}\in\Hom_{\Sq(\ca)}(\sq{X}\oplus\sh[-1]{\sq{Y}},\sq{C})$ such that the diagram
\[
\xymatrix{
 & & \sq{X}\oplus\sh[-1]{\sq{Y}} \ar[dd]^{\sq{g}} \ar[drr] \\
\sh[-1]{\sq{Y}} \ar[rru] \ar[rrd] & & & & \sq{X} \\
 & & \sq{C} \ar[rru] 
}
\]
commutes. Now, $\sq{g}\in\Hom_{\Sq(\ca)}^0(\sq{X}\oplus\sh[-1]{\sq{Y}},\sq{C})$ makes the diagram commute if and only if there exists $h^i\in\Hom_\ca(X^i,Y^{i-1})$ such that
\[
g^i=
\begin{pmatrix}
\id_{X^i} & 0 \\
h^i & \id_{Y^{i-1}}
\end{pmatrix}
\colon X^i\oplus Y^{i-1}\to C^i=X^i\oplus Y^{i-1}
\]
for every $i\in\ZZ$. On the other hand, when $\sq{g}$ is of this form, it is a morphism of $\Sq(\ca)$ if and only if $-f^i=\sm[\sq{Y}]^{i-1}\comp h^i+h^{i+1}\comp\sm[\sq{X}]^i$ for every $i\in\ZZ$. By definition, this last condition is equivalent to $\sq{f}=\sm[\SHom_{\Sq(\ca)}(\sq{X},\sq{Y})]^{-1}(-\sq{h})$.

As for the last statement, if $\Ext_\ca^1(X^i,Y^{i-1})=0$ for every $i\in\ZZ$, let $0\to\sh[-1]{\sq{Y}}\to\sq{Z}\to\sq{X}\to0$ be a short exact sequence in $\Sq(\ca)$. Up to isomorphism, we can assume that in each degree $i$ it is given by the split exact sequence $0\to Y^{i-1}\to X^i\oplus Y^{i-1}\to X^i\to0$. This easily implies that it coincides with $E_{\sq{f}}$ for some $\sq{f}\in\Hom_{\Sq(\ca)}^0(\sq{X},\sq{Y})$.
\end{proof}

\subsection{The special case of vector spaces}\label{subsec:vectsp}

Now we assume that $\kk$ is a field and set $\Sq:=\Sq(\Mod{\kk})$ and $\SQ:=\SQ(\Mod{\kk})$. We know from \autoref{lem:abGroth} that $\Sq$ is a Grothendieck abelian category. We investigate additional properties  of such a category here. In particular we describe its injective (and, in part, projective) objects.

Let us start by setting some notation. Given $a\in\ZZ\cup\{-\infty\}$ and $b\in\ZZ\cup\{\infty\}$ with $a\le b$, let $\sq{S}_{a,b}$ be the object of $\Sq$ defined by
\[
S_{a,b}^i:=
\begin{cases}
\kk & \text{if $a\le i\le b$} \\
0 & \text{otherwise}
\end{cases}
\qquad
\sm[\sq{S}_{a,b}]^i:=(-1)^i\id_\kk \quad \text{if $a\le i<b$}
\]
Unlike $a$ and $b$, our notation is that the letters $i,j,m,n$ take values in $\ZZ$.

\begin{remark}\label{isosign}
The choice of signs ensures that $\sh[n]{\sq{S}_{a,b}}=\sq{S}_{a-n,b-n}$ for every $n\in\ZZ$. For most purposes it is however irrelevant, since with a different choice one gets isomorphic objects. More generally, it is immediate to see that $\sq{V}\iso\sq{W}$ in $\Sq$ if $V^i=W^i=0$ for $i<a$ or $i>b$, $\sm[\sq{V}]^i$ and $\sm[\sq{W}]^i$ are isomorphisms for $a\le i<b$ and $V^i\iso W^i$ for some (hence for all) $i$ such that  $a\le i\le b$. 
\end{remark}

Given an object $\sq{V}$ of $\Sq$, we set $\sm^{i,j}=\sm[\sq{V}]^{i,j}:=\sm^{j-1}\comp\cdots\comp\sm^i\colon V^i\to V^j$ for $i\le j$; in particular, $\sm^{i,i}=\id_{V^i}$ and $\sm^{i,i+1}=\sm^i$. As a matter of notation we define also the following subspaces of $V^i$ for $m\le i\le n$:
\[
I^i_m=I^i_m(\sq{V}):=\im(\sm^{m,i}),\qquad K^i_n=K^i_n(\sq{V}):=\ker(\sm^{i,n}).
\]
We set moreover $K^i_\infty:=\cup_{n\ge i}K^i_n$.

\begin{remark}\label{Homindec}
For $n\le b$ and for every $\sq{V}\in\Sq$ there are natural isomorphisms in $\Mod{\kk}$
\[
\Hom_\Sq(\sq{S}_{n,b},\sq{V})\iso
\begin{cases}
V^n & \text{if $b=\infty$} \\
K^n_{b+1}(\sq{V}) & \text{if $b<\infty$}.
\end{cases}
\]
On the other hand, setting
\[
\compat{\sq{V}}:=\{(x_i)_{i\in\ZZ}\in\prod_{i\in\ZZ}V^i\st x_{i+1}=(-1)^i\sm[\sq{V}]^i(x_i)\ \all i\in\ZZ\},
\]
there is also a natural isomorphism in $\Mod{\kk}$
\[
\Hom_\Sq(\sq{S}_{-\infty,b},\sq{V})\iso
\begin{cases}
\compat{\sq{V}} & \text{if $b=\infty$} \\
\{(x_i)_{i\in\ZZ}\in\compat{\sq{V}}\st x_{b+1}=0\} & \text{if $b<\infty$}.
\end{cases}
\]
Observe that, defining $\inj{V}^j$ (for every $j\in\ZZ$) to be the image of $\compat{\sq{V}}$ through the projection $\prod_{i\in\ZZ}V^i\to V^j$, we obtain a subobject $\sq{\inj{V}}$ of $\sq{V}$ such that the natural map $\Hom_\Sq(\sq{S}_{-\infty,b},\sq{\inj{V}})\to\Hom_\Sq(\sq{S}_{-\infty,b},\sq{V})$ is an isomorphism.
\end{remark}

\begin{remark}\label{locallynoetherian}
In \autoref{Homindec} we noted
that, for every $n\in\ZZ$, we have a canonical
isomorphism, natural in $\sq{V}\in\Sq$ 
\[
\Hom_\Sq(\sq{S}_{n,\infty},\sq{V})\iso
V^n\ . 
\]
From this we learn two things:
\begin{enumerate}
\item
The functor $\Hom_\Sq(\sq{S}_{n,\infty},-)$
commutes with colimits. In particular it respects cokernels,
that is epimorphisms in the abelian category $\Sq$.
\item
The objects $\sq{S}_{n,\infty}$ \emph{generate}
the category $\Sq$, meaning that
for every $\sq{V}\in\Sq$ there exists
an epimorphism from a coproduct of
objects $\sq{S}_{n,\infty}$ to $\sq{V}$.
\setcounter{enumiv}{\value{enumi}}
\end{enumerate}
In classical terminology, the objects
$\sq{S}_{n,\infty}\in\Sq$ are a set of \emph{finitely presented,
projective generators.}

But we also observe that the only subobjects of
the object $\sq{S}_{m,\infty}$ are the objects
$\sq{S}_{n,\infty}\in\Sq$ with $m\leq n$, and
any increasing sequence of subobjects therefore
terminates. This means
\begin{enumerate}
\setcounter{enumi}{\value{enumiv}}
\item
The objects $\sq{S}_{n,\infty}$
are all \emph{noetherian} in 
the abelian category $\Sq$.
\end{enumerate}
In classical terminology: the category
$\Sq$ is a \emph{locally noetherian abelian category.}
\end{remark}

\begin{prop}\label{injchar}
An object $\sq{W}$ of $\Sq$ is injective if and only if $\sm[\sq{W}]^i$ is surjective for every $i\in\ZZ$.
\end{prop}

\begin{proof}
In the light of \autoref{Homindec} it is immediate that the $\kk$-linear map $\Hom_\Sq(\sq{S}_{i,\infty},\sq{W})\to\Hom_\Sq(\sq{S}_{i+1,\infty},\sq{W})$ induced by the natural monomorphism $\sq{S}_{i+1,\infty}\mono\sq{S}_{i,\infty}$ of $\Sq$ can be identified with $(-1)^i\sm[\sq{W}]^i\colon W^i\to W^{i+1}$, for every $i\in\ZZ$. This clearly proves that 
\begin{enumerate}
\item
The map $\sm[\sq{W}]^i$ is surjective if
and only if the functor
$\Hom_\Sq(-,\sq{W})$ takes the morphism
$\sq{S}_{i+1,\infty}\mono\sq{S}_{i,\infty}$
to an epimorphism.
\end{enumerate}
Hence if the object $\sq{W}$ is injective, then
all the maps $\sm[\sq{W}]^i$ must be epimorphisms.

Now for the converse,
assume that $\sm[\sq{W}]^i$ is surjective for every $i\in\ZZ$, and
we want to prove that
$\sq{W}$ is an injective object in $\Sq$. Here
\autoref{locallynoetherian} comes
to our aid: in the locally noetherian abelian
category $\Sq$, it suffices to prove that
$\Ext^1(\sq{A},\sq{W})=0$ vanishes for
all noetherian $\sq{A}\in\Sq$. And it even suffices
to do this where $\sq{A}$ is assumed to be a
quotient of one of our noetherian generators
$\sq{S}_{n,\infty}\in\Sq$. That is: we are reduced to proving the vanishing
of $\Ext^1(\sq{S}_{n,b},\sq{W})$, for all $n\in\ZZ$ and all $b\in[n,\infty]$.

If $b=\infty$ then the vanishing of $\Ext^1(\sq{S}_{n,b},\sq{W})$
is because the object $\sq{S}_{n,\infty}\in\Sq$ is projective;
see \autoref{locallynoetherian}(1).

Assume therefore that $b<\infty$. Then the short exact
sequence
\[\xymatrix{
0
\ar[r]
&
\sq{S}_{b+1,\infty}
\ar[r]^\alpha
&
\sq{S}_{n,\infty}
\ar[r]
&
\sq{S}_{n,b}
\ar[r]
&
0
}\]
gives a projective resolution for
$\sq{S}_{n,b}$, and applying the
functor
$\Hom(-,\sq{W})$ gives
an exact sequence
\[\xymatrix@R-10pt{
\Hom(\sq{S}_{n,\infty},,\sq{W})
\ar[r]^\alpha
&
\Hom(\sq{S}_{b+1,\infty},\sq{W})
\ar[r]^\beta
&
\Ext^1(\sq{S}_{n,b},\sq{W})
\ar[r]
&
\Ext^1(\sq{S}_{n,\infty},,\sq{W})
\ar@{=}[d] \\
& & &
0
}\]
where the vanishing is by the projectivity of
$\sq{S}_{n,\infty}$. But the map induced by
$\alpha$ is an epimorphism by (1) above, allowing
us to deduce the vanishing of
$\Ext^1(\sq{S}_{n,b},\sq{W})$.
\end{proof}

\begin{cor}\label{injpart}
For every $\sq{V}\in\Sq$ the subobject $\sq{\inj{V}}$ of $\sq{V}$ introduced in \autoref{Homindec} is the maximal injective subobject of $\sq{V}$ and is a direct summand of $\sq{V}$.
\end{cor}

\begin{proof}
It is clear by definition that $\sm[\sq{V}]^i(\inj{V}^i)=\inj{V}^{i+1}$ for every $i\in\ZZ$, whence $\sq{\inj{V}}$ is injective by \autoref{injchar}, and an injective subobject is always a direct summand. Moreover, if $\sq{U}$ is another injective subobjet of $\sq{V}$, then, again by \autoref{injchar}, $\sm[\sq{V}]^i(U^i)=U^{i+1}$ for every $i\in\ZZ$. It follows that, given $j\in\ZZ$ and $x\in U^j$ there exists $(x_i)_{i\in\ZZ}\in\compat{\sq{V}}$ such that $x_j=x$, which implies that $x\in\inj{V}^j$. Thus $\sq{U}$ is a subobject of $\sq{\inj{V}}$.
\end{proof}

\subsection{Indecomposable objects}\label{subsec:indec}

Recall that a non-zero object in an additive category is \emph{indecomposable} if it is not a coproduct of two non-zero objects. An object is \emph{completely decomposable} if it is a coproduct of indecomposable objects. Using \autoref{essbij} and the fact that $\fpr\colon\SQ\to\Sq$ preserves coproducts, we see that decompositions as a coproduct in $\Sq$ and in $\SQ$ coincide; in particular, $\Sq$ and $\SQ$ have the same indecomposable objects and the same completely decomposable objects.

\begin{remark}\label{indec}
It is clear that $\sq{S}_{a,b}$ are indecomposable in $\Sq$ (for every $a\in\ZZ\cup\{-\infty\}$ and $b\in\ZZ\cup\{\infty\}$ with $a\le b$), and we will see in \autoref{indchar} that they are the only indecomposable objects of $\Sq$, up to isomorphism.
\end{remark}

Given $A\subseteq\ZZ\cup\{-\infty\}$ and $B\subseteq\ZZ\cup\{\infty\}$, we will denote by $\Sqd[A,B]$ the (strictly) full subcategory of $\Sq$ consisting of coproducts of objects of the form $\sq{S}_{a,b}$ with $a\in A$, $b\in B$ and $a\le b$. We will write $\Sqd$ instead of $\Sqd[\ZZ\cup\{-\infty\},\ZZ\cup\{\infty\}]$. This notation will be extended to $\SQ$ in an obvious way.

\begin{lem}\label{injdec}
An object of $\Sq$ is injective if and only if it belongs to $\Sqd[\{-\infty\},\ZZ\cup\{\infty\}]$.
\end{lem}

\begin{proof}
The ``if'' part is clear by \autoref{injchar}, since each $\sq{S}_{-\infty,b}$ is injective and a coproduct of injective objects is injective.
We need to prove the converse.

Assume therefore that $\sq{V}\in\Sq$ is injective, and consider the set $A$ of all subobjects $\sq{W}\subset\sq{V}$ such that
\begin{enumerate}
\item
$\sq{W}\in\Sqd[\{-\infty\},\ZZ\cup\{\infty\}]$.
\setcounter{enumiv}{\value{enumi}}
\end{enumerate}
Define a partial order on $A$ by declaring that 
\begin{enumerate}
\setcounter{enumi}{\value{enumiv}}
\item
$\sq{W_1}<\sq{W_2}$ means that $\sq{W_1}\subset\sq{W_2}$
and $\sq{W_2}=\sq{W_1}\oplus\sq{\widetilde W}$ with
$\sq{\widetilde W}\in\Sqd[\{-\infty\},\ZZ\cup\{\infty\}]$.
\end{enumerate}
By Zorn's Lemma the set $A$ contains a maximal member, and we
assert that it must be all of $\sq{V}$.

Assume it isn't, and we prove a contradiction. Let
$\sq{W}\in A$ be maximal; being an injective subobject
of $\sq{V}$ we must have a decomposition $\sq{V}=\sq{W}\oplus\sq{\widetilde W}$,
and because $\sq{W}$ is a proper subobject of
$\sq{V}$ we must have $\sq{\widetilde W}\neq0$.
By \autoref{locallynoetherian}
there must exist a nonzero map $\sq{S}_{n,\infty}\to\sq{\widetilde W}$,
and because $\sq{\widetilde W}$ is injective the map must factor
through each of the monomorphisms
$\sq{S}_{n,\infty}\mono\sq{S}_{n-1,\infty}\mono\sq{S}_{n-2,\infty}\mono\cdots$.
Therefore it factors through the colimit, and we deduce a nonzero
map $\sq{S}_{-\infty,\infty}\to\sq{\widetilde W}$. The kernel
is a proper subobject of $\sq{S}_{-\infty,\infty}$, and the only nonzero,
proper
subobjects are $\sq{S}_{m+1,\infty}$. Hence the map must factor through
a monomorphism $\sq{S}_{-\infty,m}\mono\sq{\widetilde W}$. This
makes $\sq{W}\oplus\sq{S}_{-\infty,m}$ a subobject of
$\sq{V}$, contradicting the maximality of
$\sq{W}$.
\end{proof}

\begin{remark}\label{projchar}
In \autoref{locallynoetherian} we saw that the objects
$\sq{S}_{n,\infty}$ are projective generators for
the abelian category $\Sq$. Since coproducts of
projective objects are projective, every object in
$\Sqd[\ZZ,\{\infty\}]$ must be projective. But as
every object admits an epimorphism from an
object in $\Sqd[\ZZ,\{\infty\}]$, the projective
objects are precisely the direct summands of the
objects in $\Sqd[\ZZ,\{\infty\}]$.

It can be shown that the category
$\Sqd[\ZZ,\{\infty\}]$ is closed under direct
summands, but since we do not need it
we will not include a proof.
It would be interesting to know if
$\Sqd$ is also closed under direct summands.

It immediately follows that, if
$\sq{W}$ is projective, then $\sm[\sq{W}]^i$ is injective for every $i\in\ZZ$.
The converse, however, isn't true; the injectivity of each $\sm[\sq{W}]^i$ does not in general imply that $\sq{W}$ is projective. For instance, we claim that $\sq{S}_{-\infty,\infty}$ is not projective. In order to see this, consider the object $\sq{T}$ of $\Sq$ defined by $T^i:=\kk^\NN$ and $\sm[\sq{T}]^i(c_0,c_1,\dots):=(0,c_0,c_1,\dots)$ for every $i\in\ZZ$. As $\compat{\sq{T}}=0$, by \autoref{Homindec} we have $\Hom_\Sq(\sq{S}_{-\infty,\infty},\sq{T})=0$. On the other hand, denoting by $\sq{\quot{T}}$ the quotient of $\sq{T}$ by its subobject whose $i^{\mathrm{th}}$ component is $\kk^{(\NN)}$, it is straightforward to see that each $\sm[\sq{\quot{T}}]^i$ is an isomorphism. It follows that $\sq{T}$ is isomorphic to a non-empty (actually uncountable) coproduct of copies of $\sq{S}_{-\infty,\infty}$ (see \autoref{isosign}), whence $\Hom_\Sq(S_{-\infty,\infty},\sq{\quot{T}})\ne0$. This shows that $\sq{S}_{-\infty,\infty}$ is not projective. It can be proved that $\sq{T}$ is not projective either.
\end{remark}

\begin{lem}\label{gendec}
For every $\sq{V}\in\Sq$ the subobjects $\sq{V}_n$ of $\sq{V}$ (for $n\in\ZZ$) defined by
\[
V_n^i:=
\begin{cases}
V^i & \text{if $i\le n$} \\
I^i_n & \text{if $i>n$}
\end{cases}
\]
are such that $\sq{V}_n\subseteq\sq{V}_{n+1}$ and $\sq{V}=\cup_{n\in\ZZ}\sq{V}_n$. Moreover, $\sq{V}_n$ is a direct summand of $\sq{V}_{n+1}$.
\end{lem}

\begin{proof}
The first part being straightforward, it is enough to prove the last statement.  To this end, we define $\sq{\tilde{V}}_n\in\Sq$ by
\[
\tilde{V}_n^i:=
\begin{cases}
V^n & \text{if $i\le n$} \\
V^i_n=I^i_n & \text{if $i>n$}
\end{cases}
\qquad
\sm[\sq{\tilde{V}}_n]^i:=
\begin{cases}
\id_{V^n} & \text{if $i<n$} \\
\sm[\sq{V}_n]^i & \text{if $i\ge n$}
\end{cases}
\]
and observe that each $\sm[\sq{\tilde{V}}_n]^i$ is surjective, so that $\sq{\tilde{V}}_n$ is injective by \autoref{injchar}. Then the natural morphism $\sq{f}\colon\sq{V}_n\to\sq{\tilde{V}}_n$ (defined by $f^i:=\sm[\sq{V}]^{i,n}$ for $i<n$ and $f^i:=\id_{V^i_n}$ for $i\ge n$) extends to a morphism $\sq{g}\colon\sq{V}_{n+1}\to\sq{\tilde{V}}_n$. Finally, it is very easy to check that $\sq{h}\colon\sq{V}_{n+1}\to\sq{V}_n$ defined by $h^i:=\id_{V^i_n}$ for $i\le n$ and $h^i:=g^i$ for $i>n$ is a morphism of $\Sq$ such that $\sq{h}\rest{\sq{V}_n}=\id_{\sq{V}_n}$. 
\end{proof}

\begin{prop}\label{seqdec}
For every $\sq{V}\in\Sq$ there exists an exact sequence $0\to\sq{U}\to\sq{V}\to\sq{W}\to0$ with $\sq{U}\in\Sqd[\ZZ,\ZZ\cup\{\infty\}]$ and $\sq{W}\in\Sqd[\{-\infty\},\ZZ\cup\{\infty\}]$. Moreover, the sequence splits if, for every $i\in\ZZ$, there exists $n_i<i$ such that the inclusion $I^i_n\subseteq I^i_{n+1}$ is an equality for $n\le n_i$.
\end{prop}

\begin{proof}
By \autoref{gendec}, denoting by $\sq{U}_n$ a complementary of $\sq{V}_{n-1}$ in $\sq{V}_n$, obviously $\sq{U}:=\bigoplus_{n\in\ZZ}\sq{U}_n$ is a subobject of $\sq{V}$, and we define $\sq{W}:=\sq{V}/\sq{U}$. It is easy to see that $\sm[\sq{W}]^i$ is surjective for every $i\in\ZZ$, while $\sm[\sq{U}_n]^i$ is surjective for $i\ge n$ and $U_n^i=0$ for $i<n$. It follows from \autoref{injchar} that $\sq{W}$ is injective and that there exist injective objects $\sq{\tilde{U}}_n$ such that $\sq{U}_n=\tilde{U}_n^{\ge n}$ (one can take, for instance, $\tilde{U}_n^i:=U_n^n$ and $\sm[\sq{\tilde{U}}_n]^i:=\id_{U^n_n}$ for $i<n$). By \autoref{injdec} we obtain, as wanted, that $\sq{W}\in\Sqd[\{-\infty\},\ZZ\cup\{\infty\}]$ and $\sq{U}_n\in\Sqd[\{n\},\ZZ\cup\{\infty\}]$ for every $n\in\ZZ$.

As for the last statement, observe that, in any case, $\sq{V}_{-\infty}:=\cap_{n\in\ZZ}\sq{V}_n$ is a subobject of $\sq{V}$ such that $\sq{U}\cap\sq{V}_{-\infty}=0$. Assuming that, for every $i\in\ZZ$, there exists $n_i<i$ such that $I^i_n=I^i_{n+1}$ for $n\le n_i$, it is clear that $V^i_{-\infty}=I^i_n$  for $n\le n_i$ and $V^i=U^i+V^i_{-\infty}$. Thus we conclude that $\sq{V}=\sq{U}\oplus\sq{V}_{-\infty}$ (hence $\sq{W}\iso\sq{V}_{-\infty}$) in this case.
\end{proof}

\begin{cor}\label{indchar}
An object $\sq{V}$ of $\Sq$ is indecomposable (respectively, completely decomposable) if and only if $\sq{V}\iso\sq{S}_{a,b}$ for some $a\in\ZZ\cup\{-\infty\}$ and $b\in\ZZ\cup\{\infty\}$ with $a\le b$ (respectively, $\sq{V}\in\Sqd$).
\end{cor}

\begin{proof}
With the notation used in the proof of \autoref{seqdec}, $\sq{V}=(\bigoplus_{m>n}\sq{U}_m)\oplus\sq{V}_n$, for every $\sq{V}\in\Sq$ and every $n\in\ZZ$.

Now, if $\sq{V}$ is indecomposable, then either $\sq{V}_n=\sq{V}$ for every $n\in\ZZ$ or there exists $n\in\ZZ$ such that $\sq{V}_n=\sq{V}$ and $\sq{V}_{n-1}=0$. In the former case $\sq{U}=\bigoplus_{n\in\ZZ}\sq{U}_n=0$, whence $\sq{V}\iso\sq{W}\in\Sqd[\{-\infty\},\ZZ\cup\{\infty\}]$, whereas in the latter case $\sq{V}=\sq{U}_n\in\Sqd[\{n\},\ZZ\cup\{\infty\}]$. As $\sq{V}$ is indecomposable, we obtain $\sq{V}\iso\sq{S}_{a,b}$ for some $a\le b$. By definition, this implies that the completely decomposable objects of $\Sq$ are precisely those of $\Sqd$.
\end{proof}

\section{The derived category of dual numbers}\label{dercatdual}

In this section we investigate several basic properties of various triangulated categories associated to categories of modules over the dual numbers $\kep$, where $\kk$ is a field.

\subsection{Homotopy categories and their relatives}\label{subsec:homotpy}

As a preliminary step, let us consider the following full subcategories of the homotopy category $\K:=\K(\Mod{\kep})$. First recall the following.

\begin{remark}\label{freemod}
(i) It is a standard property that a $\kep$-module is free if and only if it is projective if and only if it is injective.

(ii) If $M$ is a free $\kep$-module, then $M\iso V\otimes_\kk\kep$ with $V:=M\otimes_{\kep}\kk\in\Mod{\kk}$. If $N\iso W\otimes_\kk\kep$ (for some $W\in\Mod{\kk}$) is another free $\kep$-module, every morphism 
\[
f\in\Hom_{\kep}(M,N)\iso\Hom_{\kep}(V\otimes_\kk\kep,W\otimes_\kk\kep)\iso\Hom_\kk(V,W\otimes_\kk\kep)
\]
can be represented (uniquely) as $f_1+\ep f_\ep$ with $f_1,f_\ep\in\Hom_\kk(V,W)$. With this notation, given another morphism of free $\kep$-modules $g\colon N\to P$, it is obvious that $(g\comp f)_1=g_1\comp f_1$ and $(g\comp f)_\ep=g_1\comp f_\ep+g_\ep\comp f_1$.
\end{remark}

We can then define the following full subcategories of $\K$:
\begin{itemize}
\item $\Kac$ with objects the acyclic complexes;
\item $\Khp$ with objects the h-projective complexes (namely, the left orthogonal of $\Kac$ in $\K$);
\item $\Kco$ with objects the cofibrant complexes for the projective model structure on $\C:=\C(\Mod{\kep})$;
\item $\Kfr$ with objects the complexes $\sq{M}$ such that $M^i$ is a free $\kep$-module for every $i\in\ZZ$ (it is the homotopy category of injectives of $\Mod{\kep}$);
\item $\Kmi$ with objects $\sq{M}\in\Kfr$ which are \emph{minimal}, meaning that $\sm[\sq{M}]^i\in\ep\Hom_{\kep}(M^i,M^{i+1})$ for every $i\in\ZZ$;
\item $\Kaf:=\Kac\cap\Kfr$;
\item $\Kam:=\Kac\cap\Kmi$;
\item $\Kpf:=\Khp\cap\Kfr$;
\item $\Kpm:=\Khp\cap\Kmi$.
\end{itemize}

A preliminary comparison between the categories above is provided by the following result.

\begin{prop}\label{minimal}
The inclusion $\Kmi\subseteq\Kfr$ is an equivalence.
\end{prop}

\begin{proof}
Given $\sq{M}\in\Kfr$, by \autoref{freemod} (ii) we can assume $M^i=V^i\otimes_\kk\kep$ (with $V^i\in\Mod{\kk}$) and $\sm^i=\sm[\sq{M}]^i$ is represented by $\sm[1]^i+\ep\sm[\ep]^i$ (with $\sm[1]^i,\sm[\ep]^i\in\Hom_\kk(V^i,V^{i+1})$), for every $i\in\ZZ$. Moreover, $\sm^{i+1}\comp\sm^i=0$ is equivalent to
\begin{gather}
\sm[1]^{i+1}\comp\sm[1]^i=0,\label{diff1} \\
\sm[1]^{i+1}\comp\sm[\ep]^i+\sm[\ep]^{i+1}\comp\sm[1]^i=0.\label{diffep}
\end{gather}
By \eqref{diff1} $B^i:=\im(\sm[1]^{i-1})\subseteq Z^i:=\ker(\sm[1]^i)$. Since $B^{i+1}\iso V^i/Z^i$, setting $H^i:=Z^i/B^i$, we can clearly assume $V^i=B^i\oplus H^i\oplus B^{i+1}$ and
\[
\sm[1]^i=
\begin{pmatrix}
0 & 0 & \id_{B^{i+1}} \\
0 & 0 & 0 \\
0 & 0 & 0
\end{pmatrix}
\colon V^i=B^i\oplus H^i\oplus B^{i+1}\to V^{i+1}=B^{i+1}\oplus H^{i+1}\oplus B^{i+2}.
\]
Then it is easy to see that \eqref{diffep} is satisfied if and only if $\sm[\ep]^i$ is of the form
\[
\sm[\ep]^i=
\begin{pmatrix}
e^i & a^i & b^{i+1} \\
0 & h^i & c^{i+1} \\
0 & 0 & -e^{i+1}
\end{pmatrix}
\colon V^i=B^i\oplus H^i\oplus B^{i+1}\to V^{i+1}=B^{i+1}\oplus H^{i+1}\oplus B^{i+2}
\]
(with arbitrary $\kk$-linear maps $a^i$, $b^i$, $c^i$, $e^i$ and $h^i$), for every $i\in\ZZ$. We claim that $\sq{M}$ is homotopy equivalent to $\sq{N}\in\Kmi$ defined by $N^i:=H^i\otimes_\kk\kep$ and $\sm[\sq{N}]^i$ represented by $\ep h^i$, for every $i\in\ZZ$. To this end, we must find morphisms $\sq{f}\colon\sq{M}\to\sq{N}$ and $\sq{g}\colon\sq{N}\to\sq{M}$ of $\C$ inducing inverse isomorphisms in $\K$. If $f^i$ and $g^i$ are represented, respectively, by $f^i_1+\ep f^i_\ep$ and $g^i_1+\ep g^i_\ep$ (with $f^i_1,f^i_\ep\in\Hom_\kk(V^i,H^i)$ and $g^i_1,g^i_\ep\in\Hom_\kk(H^i,V^i)$), we can set
\begin{gather*}
f^i_1:=
\begin{pmatrix}
0 & \id_{H^i} & 0
\end{pmatrix},
f^i_\ep:=
\begin{pmatrix}
-c^i & 0 & 0
\end{pmatrix}
\colon V^i=B^i\oplus H^i\oplus B^{i+1}\to H^i, \\
g^i_1:=
\begin{pmatrix}
0 \\
\id_{H^i} \\
0
\end{pmatrix},
g^i_\ep:=
\begin{pmatrix}
0 \\
0 \\
-a^i
\end{pmatrix}
\colon H^i\to V^i=B^i\oplus H^i\oplus B^{i+1}. \\
\end{gather*}
Indeed, it is straightforward to check that $\sm[\sq{N}]^i\comp f^i=f^{i+1}\comp\sm[\sq{M}]^i$, namely that
\[
0=f^{i+1}_1\comp\sm[1]^i, \qquad
h^i\comp f^i_1=f^{i+1}_1\comp\sm[\ep]^i+f^{i+1}_\ep\comp\sm[1]^i,
\]
and that $\sm[\sq{M}]^i\comp g^i=g^{i+1}\comp\sm[\sq{N}]^i$, namely that
\[
\sm[1]^i\comp g^i=0, \qquad
\sm[1]^i\comp g^i_\ep+\sm[\ep]^i\comp g^i_1=g^{i+1}_1\comp h^i.
\]
Finally, it is easy to verify that $\sq{f}\comp\sq{g}=\id_{\sq{N}}$, whereas a homotopy between $\sq{g}\comp\sq{f}$ and $\id_{\sq{M}}$ is given by morphisms $k^i\colon M^i\to M^{i-1}$ of $\Mod{\kep}$, represented by $k^i_1+k^i_\ep$, where
\[
k^i_1:=
\begin{pmatrix}
0 & 0 & 0 \\
0 & 0 & 0 \\
\id_{B^i} & 0 & 0
\end{pmatrix},
k^i_\ep:=
\begin{pmatrix}
0 & 0 & 0 \\
0 & 0 & 0 \\
-b^i & 0 & 0
\end{pmatrix}
\colon V^i=B^i\oplus H^i\oplus B^{i+1}\to V^{i-1}=B^{i-1}\oplus H^{i-1}\oplus B^i,
\]
for every $i\in\ZZ$.
\end{proof}

\begin{remark}\label{pmpf}
Clearly $\Kfr$ is a (non strictly) full triangulated subcategory of $\K$, and \autoref{minimal} implies that the same is true for $\Kmi$. Moreover, since $\Kac$ and $\Khp$ are localizing subcategories of $\K$, we obtain also that $\Kam$ and $\Kpm$ are localizing subcategories of $\Kmi$. Finally, again as a direct consequence of \autoref{minimal}, we see that the inclusions $\Kam\subseteq\Kaf$ and $\Kpm\subseteq\Kpf$ are equivalences, as well.
\end{remark}

\subsection{Categories of sequences}\label{subsec:sequences}

Keeping the notation from the previous section, we can now prove the following.

\begin{prop}\label{KmiSQ}
There is a natural equivalence between $\Kmi$ and $\SQ$, which is compatible with shifts (hence the latter category, with the standard shift functor, is triangulated in a natural way).
\end{prop}

\begin{proof}
By \autoref{freemod} every object $\sq{M}\in\Kmi$ corresponds to an object $\sq{V}\in\Sq$ such that $V^i:=M^i\otimes_{\kep}\kk$ and $\sm[\sq{M}]^i$ is represented by $\ep\sm[\sq{V}]^i$, for every $i\in\ZZ$. On the other hand, given also $\sq{N}\in\Kmi$ corresponding to $\sq{W}\in\Sq$, a morphism $\sq{f}\colon\sq{M}\to\sq{N}$ in $\C$ can be identified with a sequence $\{f^i_1+\ep f^i_\ep\}_{i\in\ZZ}$ with $f^i_1,f^i_\ep\in\Hom_\kk(V^i,W^i)$ such that $\sq{f_1}\in\Hom_\Sq(\sq{V},\sq{W})$. Moreover such a morphism is homotopic to $0$ if and only if $\sq{f_1}=0$ and $\sq{f_\ep}\in\im(\sm[\SHom_\Sq(\sq{V},\sq{W})]^{-1})$. It follows that there is a natural isomorphism $\Hom_{\Kmi}(\sq{M},\sq{N})\iso\Hom_{\SQ}(\sq{V},\sq{W})$. It is also clear that compositions and shifts in the two categories are identified.
\end{proof}

Using the equivalence above, we can actually define the full subcategories of $\SQ$:
\begin{itemize}
\item $\SQa$ corresponding to $\Kam$;
\item $\SQp$ corresponding to $\Kpm$.
\end{itemize}

\begin{remark}\label{triangles}
As in every triangulated category, a morphism $h\colon\sq{V}\to\sq{W}$ of $\SQ$ can be extended to a distinguished triangle
\[
\sh[-1]{\sq{W}}\mor{f}\sq{U}\mor{g}\sq{V}\mor{h}\sq{W}
\]
(which is unique up to isomorphism). Using (beyond the equivalence of \autoref{KmiSQ}) the explicit computations in the proof of \autoref{minimal} in order to replace the mapping cone in $\Kfr$ with an isomorphic object of $\Kmi$, it is not hard to prove that such a triangle satisfies the following properties. If $h=\sq{h}_1+[\sq{h}_\ep]$, then $U^i=\ker(h^i_1)\oplus\cok(h^{i-1}_1)$ and
\[
\sm[\sq{U}]^i=
\begin{pmatrix}
\alpha^i & 0 \\
-\gamma^i & -\beta^{i-1} 
\end{pmatrix}
\colon U^i=\ker(h^i_1)\oplus\cok(h^{i-1}_1)\to U^{i+1}=\ker(h^{i+1}_1)\oplus\cok(h^i_1),
\]
where $\alpha^i$ is the restriction of $\sm[\sq{V}]^i$, $\beta^{i-1}$ is induced by $\sm[\sq{W}]^{i-1}$ and $\gamma^i$ is the composition $\ker(h^i_1)\mono V^i\mor{h^i_\ep}W^i\epi\cok(h^i_1)$. On the other hand, $f^i_1$ and $g^i_1$ are the natural morphisms, whereas $f^i_\ep$ and $g^i_\ep$ factor, respectively, through $\im(h^{i-1}_1)$ and $\im(h^i_1)$.

In particular, if $h=[\sq{h}_\ep]$, then $f=\sq{f}_1$, $g=\sq{g}_1$ and there is a short exact sequence
\[
0\to\sh[-1]{\sq{W}}\mor{f}\sq{U}\mor{g}\sq{V}\to0
\]
in $\Sq$. One can also check that the element of $\Ext^1_\Sq(\sq{V},\sh[-1]{\sq{W}})$ defined by the isomorphism class of this sequence corresponds to $h\in\Homep_\Sq(\sq{V},\sq{W})$ under the isomorphism of \autoref{Ext1}. Conversely, it should then be clear that every short exact sequence of $\Sq$ extends to a distinguished triangle of $\SQ$ by adding, as a third morphism, the one corresponding to the isomorphism class of the sequence.
\end{remark}

\subsection{Derived categories}\label{subsec:derived}

Set now $\Da:=\Da(\Mod{\kep})$, for $?=sb,b,+,-,\emptyset$. As the notation is certainly not standard, we should explain what is the `strictly bounded' derived category $\Dsb$. We refer to \cite[Example 8.4 and Definition 8.5]{Nee} for its definition in a more general setting. Here, it suffices to say that, when $\ca$ is an abelian category with  a finitely-presented projective generator, then $\Dsb(\ca)$ is the full subcategory of $\Dm(\ca)$ corresponding to $\Kb(\Proj(\ca))$ under the natural equivalence between $\Dm(\ca)$ and $\Km(\Proj(\ca))$. Here $\Proj(\ca)$ denotes the full subcategory of $\ca$ consisting of projective objects. Observe that $\Dsb(\ca)\subseteq\Db(\ca)$, with equality if and only if $\ca$ has finite homological dimension.

We also denote by $\Dafg$, for $?=b,-$,  the full subcategory of $\Da$ whose objects have cohomologies in $\fMod{\kep}$ (which can be identified with $\Da(\fMod{\kep})$) and by $\Dcp=\Dp(\kep)$ the full subcategory of compact objects of $\D$.

\begin{lem}\label{KD}
The natural functor $\Kpm\to\D$ is an exact equivalence. Moreover, there is a semiorthogonal decomposition $\Kmi=\ort{\Kam,\Kpm}$.
\end{lem}

\begin{proof}
It is well known that there is a semiorthogonal decomposition
\begin{equation}\label{semiort}
\K=\ort{\Kac,\Khp},
\end{equation}
whence the composition of the natural (exact) functors $\Khp\to\K\to\D=\K/\Kac$ is an equivalence. Taking into account that $\Kco\subseteq\Kpf$ (by \cite[Lemma 2.3.6, 2.3.8]{Hov}) and that the inclusion $\Kco\subseteq\Khp$ is an equivalence (because for every $\sq{M}\in\Khp$ a cofibrant replacement $\sq{N}\to\sq{M}$ is a quasi-isomorphism in $\Khp$, and so it is an isomorphism), we deduce that the inclusion $\Kpf\subseteq\Khp$ is an equivalence. This concludes the proof of the first statement, since, by \autoref{pmpf}, also the inclusion $\Kpm\subseteq\Kpf$ is an equivalence. As for the last statement, it is clear that $\Kpm$ is contained in the left orthogonal of $\Kam$. Moreover, given $\sq{M}\in\Kmi$, by \eqref{semiort} there is a distinguished triangle $\sq{P}\to\sq{M}\to\sq{N}$ with $\sq{P}\in\Khp$ and $\sq{N}\in\Kac$. Actually we can assume $\sq{P}\in\Kpm$ (because the inclusion $\Kpm\subseteq\Khp$ is an equivalence), and then also $\sq{N}\in\Kam$.
\end{proof}

Combining \autoref{KD} and \autoref{KmiSQ}, we obtain the following result.

\begin{cor}\label{DSQp}
We have a semiorthogonal decomposition $\SQ=\ort{\SQa,\SQp}$. Moreover, there is an exact equivalence between $\SQp$ and $\D$.
\end{cor}

We will also denote by $\SQpa$, $\SQpafg$ and $\SQpc$ the strictly full (triangulated) subcategories of $\SQp$ corresponding (under the equivalence of \autoref{DSQp}) to $\Da$, $\Dafg$ and $\Dcp$, respectively. Moreover, $\Sqa$ will be the full subcategory of $\Sq$ with the same objects as $\SQa$; similarly for $\Sqpa$, $\Sqpafg$ and $\Sqpc$.

\begin{remark}\label{Sqachar}
An object $\sq{V}$ of $\Sq$ is in $\Sqa$ if and only if $\sm[\sq{V}]^i$ is an isomorphism, for every $i\in\ZZ$. More generally, it is easy to see that, if $\sq{V}$ corresponds to $\sq{M}\in\Kmi$, then $H^i(\sq{M})\iso\ker(\sm[\sq{V}]^i)\oplus\cok(\sm[\sq{V}]^{i-1})$ in $\Mod{\kk}$, for every $i\in\ZZ$.
\end{remark}

We can also prove the following.

\begin{prop}\label{Sqpchar}
An object $\sq{V}$ of $\Sq$ is in $\Sqp$ if and only if $K^i_\infty=V^i$ for every $i\in\ZZ$.
\end{prop}

\begin{proof}
By \autoref{DSQp} $\sq{V}\in\Sqp$ if and only if $\Hom_{\SQ}(\sq{V},\sq{W})=0$ for every $\sq{W}\in\Sqa$. Now, by \autoref{Sqachar} $\sq{W}\in\Sqa$ if and only if $\sm[\sq{W}]^i$ is an isomorphism for every $i\in\ZZ$. When $\sq{W}$ is of this form, $\sh[-1]{\sq{W}}$ in injective in $\Sq$ by \autoref{injchar}, and so $\Ext^1_\Sq(\sq{V},\sh[-1]{\sq{W}})=0$. Remembering \autoref{Ext1}, this means $\Hom_{\SQ}(\sq{V},\sq{W})=\Hom_\Sq(\sq{V},\sq{W})$.

Suppose first that $K^i_\infty(\sq{V})=V^i$ for every $i\in\ZZ$. We have to prove that, if $\sq{f}\colon\sq{V}\to\sq{W}$ is a morphism of $\Sq$ with $\sq{W}\in\Sqa$, then $f^i=0$ for every $i\in\ZZ$. This is true because obviously $f^i(K^i_n(\sq{V}))\subseteq K^i_n(\sq{W})=0$ (for $n\ge i$), whence $f^i(V^i)=f^i(K^i_\infty(\sq{V}))\subseteq K^i_\infty(\sq{W})=0$.

Conversely, suppose that $\sq{V}\in\Sq$ is some
object, $i\in\ZZ$ is an integer,
and $x\in V^i$ \emph{does not belong to}
$K^i_\infty(\sq{V})$. We need to show
$\sq{V}\notin\Sqp$. But by
\autoref{Homindec} the element
$x\in V^i$ corresponds to a map
$\chi:\sq{S}_{i,\infty}\to V$, and the
assumption that
$x\notin K^i_\infty(\sq{V})$ means that
the map $\chi:\sq{S}_{i,\infty}\to V$ is a monomorphism.
Because the object
$\sq{S}_{-\infty,\infty}\in\Sq$
is injective,
the (monomorphism)
$\sq{S}_{i,\infty}\mono\sq{S}_{-\infty,\infty}$
factors through $\chi$,
to give a nonzero map $\sq{V}\to\sq{S}_{-\infty,\infty}$
with $\sq{S}_{-\infty,\infty}\in\Sqa$.
\end{proof}

\begin{remark}\label{Sqpachar}
Using \autoref{Sqachar} it is immediate to show that $\sq{V}\in\Sqp$ is in $\Sqpu$ (respectively, in $\Sqpm$) if and only if there exists $n\in\ZZ$ such that $\sm^i$ is an isomorphism for every $i<n$ (respectively, $V^i=0$ for every $i>n$). Moreover, $\sq{V}\in\Sqpa$, for $?=b,-$, is in $\Sqpafg$ if and only if $\dim_\kk(V^i)<\infty$ for every $i\in\ZZ$. Finally, $\sq{V}\in\Sqpb$ (respectively, $\sq{V}\in\Sqpbfg$) is in $\Sqpsb$ (respectively, $\Sqpc$) if and only if there exists $n\in\ZZ$ such that $V^i=0$ for every $i<n$.
\end{remark}

Let us reconsider the characterization of indecomposable objects in \autoref{subsec:indec}. We keep the notation introduced there.

\begin{remark}\label{indec2}
By \autoref{Sqpchar} and \autoref{Sqpachar} $\sq{S}_{a,b}\in\Sqp$ if and only if $\sq{S}_{a,b}\in\Sqpbfg$ if and only if $b<\infty$, whereas $\sq{S}_{a,b}\in\Sqpc$ if and only if $a>-\infty$ and $b<\infty$. Observe also that, denoting by $\cc$ any of the subcategories $\Sqpa$, $\Sqpafg$ or $\Sqpc$ of $\Sq$, clearly $\cc$ is closed under direct summands in $\Sq$ and the inclusion of $\cc$ in $\Sq$ preserves the coproducts which exist in $\cc$. It follows that an object of $\cc$ is indecomposable (respectively, completely decomposable) in $\cc$ if and only if it is indecomposable (respectively, completely decomposable) in $\Sq$. 
\end{remark}

Let us set $\Sqpad:=\Sqpa\cap\Sqd$, for $?=sb,b,+,-,\emptyset$, and $\Sqpafgd:=\Sqpafg\cap\Sqd$, for $?=b,-$. This notation will be extended to $\SQ$ in an obvious way.

\begin{lem}\label{indchar2}
We have the equalities $\Sqpad=\Sqpa$, for $?=sb,b,+$, and $\Sqpafgd=\Sqpafg$, for $?=b,-$.
\end{lem}

\begin{proof}
If $\sq{V}\in\Sqpu$ (respectively, $\sq{V}\in\Sqpmfg$), then, by \autoref{Sqpachar}, $\sm^n$ is an isomorphism for $n\ll0$ (respectively, $\dim_\kk(V^i)<\infty$ for every $i\in\ZZ$). In both cases this clearly implies that, for every $i\in\ZZ$, there exists $n_i<i$ such that $I^i_n=I^i_{n+1}$ for $n\le n_i$. We conclude by \autoref{seqdec} that $\sq{V}\iso\sq{U}\oplus\sq{W}\in\Sqd$.
\end{proof}

\begin{remark}\label{nondec}
We have $\Sqpmd\subsetneq\Sqpm$ (hence also $\Sqpd\subsetneq\Sqp$). Indeed, with the notation of \autoref{projchar}, $T^{\le n}\in\Sqpm$ for every $n\in\ZZ$ by \autoref{Sqpachar}, but we claim that $T^{\le n}\not\in\Sqd$ (observe that this implies also $\sq{T}\not\in\Sqd$). The statement being independent of $n$, we assume on the contrary $T^{\le0}\in\Sqd$. By \autoref{indchar} there exist $V_{a,b}\in\Mod{\kk}$ for $a\le b\le0$ such that $T^{\le0}\iso\bigoplus_{-\infty\le a\le b\le0}(\sq{S_{a,b}}\otimes_\kk V_{a,b})$, which implies $\kk^\NN=T^0\iso\bigoplus_{-\infty\le a\le0}V_{a,0}$. Now, for every $m\in\NN$ we have
\[
\bigoplus_{-m<a\le0}V_{a,0}\iso T^0/I^0_{-m}=\kk^\NN/\kk^{\NN_{\ge m}}\iso\kk^m,
\]
whereas $V_{-\infty,0}=0$ (because $\bigcap_{m\in\NN}I^0_{-m}=0$). This gives the contradiction that $\kk^\NN\iso\bigoplus_{-\infty\le a\le0}V_{a,0}$ has countable dimension over $\kk$.
\end{remark}

\section{Dg categories and enhancements}\label{sec:enh}

This section collects some basic notions related to dg categories and incorporates a discussion about uniqueness and strong uniqueness of enhancements for triangulated categories of geometric significance.

\subsection{Definitions}\label{subsec:enhdefex}

In full generality, in this section we assume that $\kk$ is a commutative ring.

\begin{definition}\label{def:dg}
(i) A \emph{dg category} is a $\kk$-linear category $\cc$ whose morphism spaces $\Hom_\cc\left(A,B\right)$ are complexes of $\kk$-modules and the composition maps $\Hom_\cc(B,C)\otimes_{\kk}\Hom_\cc(A,B)\to\Hom_\cc(A,C)$ are morphisms of complexes, for all $A,B,C$ in $\cc$.

(ii) A \emph{dg functor} $\fF\colon\cc_1\to\cc_2$ between two dg categories is a $\kk$-linear functor such that the maps $\Hom_{\cc_1}(A,B)\to\Hom_{\cc_2}(\fF(A),\fF(B))$ are morphisms of complexes, for all $A,B$ in $\cc_1$.
\end{definition}

For a dg category $\cc$, we can form its \emph{homotopy category} $H^0(\cc)$ which has the same objects as $\cc$ while $\Hom_{H^0(\cc)}(A,B):=H^0(\Hom_\cc(A,B))$, for all $A,B$ in $\cc$. A dg functor $\fF\colon\cc_1\to\cc_2$ induces a $\kk$-linear functor $H^0(\fF)\colon H^0(\cc_1)\to H^0(\cc_2)$ and a dg functor $\fF$ is a \emph{quasi-equivalence} if the maps $\Hom_{\cc_1}(A,B)\to\Hom_{\cc_2}(\fF(A),\fF(B))$ are quasi-isomorphisms, for all $A,B$ in $\cc_1$, and $H^0(\fF)$ is an equivalence.

\begin{ex}\label{ex:dgcompl}
(i) Assume that $\ca$ is a $\kk$-linear category. We can form the dg category $\dgC(\ca)$, whose objects are (unbounded) complexes of objects in $\ca$. As graded modules, morphisms are defined as
\[
\Hom_{\dgC(\ca)}(A,B)^n:=\prod_{i\in\ZZ}\Hom_\ca(A^i,B^{n+i})
\]
for every $A,B\in\dgC(\ca)$ and for every $n\in\ZZ$. The composition of morphisms is the obvious one, while the differential is defined on a homogeneous element $f\in\Hom_{\dgC(\ca)}(A,B)^n$ by $d(f):=d_B\comp f-(-1)^nf\comp d_A$. Clearly we can similarly define $\dgC^?(\ca)$, for $?=b,+,-,\emptyset$, by imposing the obvious bounds to the complexes.

(ii) As dg functors between two dg categories $\cc_1$ and $\cc_2$ form a dg category $\dgHom(\cc_1,\cc_2)$, for every dg category $\cc$ we can construct a much larger $\dgMod{\cc}:=\dgHom(\cc\opp,\dgC(\Mod{\kk}))$ whose objects are called \emph{(right) dg $\cc$-modules}. The category $H^0(\dgMod{\cc})$ has a natural triangulated structure (see \cite[Section 2.2]{K3}). On the other hand, given a dg category $\cc$, the map defined on objects by $A\mapsto\Hom_\cc(\farg,A)$ extends to a fully faithful dg functor $\dgYon\colon\cc\to\dgMod{\cc}$ called the \emph{dg Yoneda embedding}.
\end{ex}

Let us denote by $\dgCat$ the category of all (small) dg categories (linear over $\kk$). Its localization, which is obtained by inverting all quasi-equivalences in $\dgCat$, will be denoted by $\Hqe$. Note that, if $\cc_1$ and $\cc_2$ are dg categories and $f\colon\cc_1\to\cc_2$ is a morphism in $\Hqe$, then one gets an induced functor $H^0(f)\colon H^0(\cc_1)\to H^0(\cc_2)$, which is well defined up to isomorphism. Note that, if $f$ is an isomorphism in $\Hqe$, then $H^0(f)$ is an equivalence.

\begin{definition}\label{def:dgpreptr}
(i) A dg category $\cc$ is \emph{pretriangulated} if the essential image of $H^0(\dgYon)$ is a full triangulated subcategory of $H^0(\dgMod{\cc})$.

(ii) Given a triangulated category $\ct$ an \emph{enhancement} of $\ct$ is a pair $(\cc,\fF)$, where $\cc$ is a pretriangulated dg category and $\fF\colon H^0(\cc)\to\ct$ is a triangulated equivalence.

(iii) A triangulated category $\ct$ is \emph{algebraic} if it has an enhancement.
\end{definition}

\begin{ex}\label{ex:models}
(i) If $\ca$ is an additive ($\kk$-linear) category, it is easy to see that $\dgC^?(\ca)$ is pretriangulated and there is a natural triangulated equivalence $H^0(\dgC^?(\ca))\iso\K^?(\ca)$. Thus the triangulated category $\K^?(\ca)$ is algebraic, for $?=b,+,-,\emptyset$.

(ii) Assume now that $\ca$ is an abelian category, and denote by $\Acya(\ca)\subseteq\Ka(\ca)$ the full triangulated subcategory consisting of acyclic complexes. Then the derived category $\D^?(\ca)$ of $\ca$, which is the Verdier quotient $\Ka(\ca)/\Acya(\ca)$, is algebraic. Here, as usual, $?=b,+,-,\emptyset$. Indeed, it is true more generally that, if $\ct$ is an algebraic triangulated category and $\cs$ is a full triangulated subcategory of $\ct$, then $\ct/\cs$ is algebraic, as well. Explicitly, an enhancement of $\ct/\cs$ can be obtained as the \emph{Drinfeld quotient} (see \cite{Dr}) of an enhancement of $\ct$ by the induced (by restriction) enhancement of $\cs$.
\end{ex}

\subsection{Uniqueness}\label{subsec:enhuniq}

For an algebraic triangulated category, it is a natural question to ask how many dg enhancements it possesses. First of all, we want to clarify in which sense they may be unique.

\begin{definition}\label{def:dguniq}
Let $\ct$ be an algebraic triangulated category.

(i) The category $\ct$ \emph{has a unique enhancement} if, given two enhancements $(\cc_1,\fF_1)$ and $(\cc_2,\fF_2)$ of $\ct$, the dg categories $\cc_1$ and $\cc_2$ are isomorphic in $\Hqe$.

(ii) The category $\ct$ \emph{has a strongly unique enhancement} if, given two enhancements $(\cc_1,\fF_1)$ and $(\cc_2,\fF_2)$ of $\ct$, there exists an isomorphism  $f\colon\cc_1\to\cc_2$ in $\Hqe$ such that $\fF_2\comp H^0(f)\iso\fF_1$.
\end{definition}

The more general result about uniqueness of enhancements which is now available in the literature is the following.

\begin{thm}[\cite{CNS2}, Theorems A and B]\label{thm:main2}
{\rm (i)} If $\ca$ is an abelian category, then $\D^?(\ca)$ has a unique enhancement, for $?=b,+,-,\emptyset$.

{\rm (ii)} Let $X$ be a quasi-compact and quasi-separated scheme. Then the categories $\D_\mathbf{qc}^?(X)$ and $\Dp(X)$ have a unique dg enhancement, for $?=b,+,-,\emptyset$.
\end{thm}

Here $\D_\mathbf{qc}^?(X)$ denotes the full subcategory of $\D_\mathbf{qc}(X)$ which consists of complexes of $\mathcal{O}_X$ modules with quasi-coherent cohomology. On the other hand, $\Dp(X)$ is the full subcategory of  $\D_\mathbf{qc}(X)$ consisting of \emph{perfect} complexes (i.e.\ complexes in $\D_\mathbf{qc}(X)$ which are locally quasi-isomorphic to bounded complexes of locally free sheaves of finite rank). Alternatively, $\Dp(X)$ coincides with the category $\D_\mathbf{qc}(X)^c$ of compact objects in $\D_\mathbf{qc}(X)$. Recall that, for a triangulated category $\ct$, an object $C\in\ct$ is  
\emph{compact} if $\Hom_\ct(C,-)$ respects those coproducts that exist in $\ct$.

It must be said that \autoref{thm:main2} comes at the end of a long story. Indeed, in \cite{BLL} it was conjectured that, for a smooth projective scheme over a field, the bounded derived category of coherent sheaves should have a unique enhancement. The conjecture was proved in \cite{LO}and these results were then extended in \cite{CS6}. Part of the results in the latter paper have been reproved in \cite{GR}. Finally \cite{A1} proposed a different approach to the problem by means of techniques which use the theory of stable $\infty$-categories.

As for strong uniqueness of enhancements, the first results were obtained in \cite{LO}, and later in a series of papers \cite{CS1}, \cite{Ol} and \cite{Lo}). But, despite some special cases, the general picture about strongly uniqueness remains elusive.

\subsection{Strong uniqueness for hereditary categories}\label{subsec:stronguniqher}

In this section we cover yet another case of strong uniqueness of enhancements: the derived category of a hereditary category, while in the next section we prove our main results concerning various triangulated categories associated to the category of modules over the dual numbers.

Let us first recall the following.

\begin{definition}\label{def:hereditary}
An abelian category $\cb$ is \emph{hereditary}, if $\Ext_\cb^2(A,B)=0$, for all $A,B\in\cb$.   
\end{definition}

We have the following examples of such abelian categories.

\begin{ex}\label{ex:her}
(i) If $R$ is a left hereditary ring (meaning that all left ideals of $R$ are projective $R$-modules), then the abelian category $\Mod{R}$ of left $R$-modules is hereditary. The same is true for $\fMod{R}$ if $R$ is also left noetherian.

(ii) If $C$ is a smooth curve over a field, then the abelian categories $\Coh(C)$ of coherent sheaves and $\Qcoh(C)$ of quasi-coherent sheaves on $C$ are hereditary. 

(iii) The abelian category $\Sq$ is hereditary. Indeed, $\Sq$ has enough injectives (being a Grothendieck category), hence every object $\sq{V}$ of $\Sq$ can be embedded in a short exact sequence $0\to\sq{V}\to\sq{W}\mor{\sq{g}}\sq{U}\to0$ with $\sq{W}$ injective. Since $\sm[\sq{W}]^i$ (by \autoref{injchar}) and $g^{i+1}$ are surjective, also $\sm[\sq{U}]^i$ is surjective, for every $i\in\ZZ$. This proves that $\sq{U}$ is injective (again by \autoref{injchar}), whence $\sq{V}$ has an injective resolution of length $\le1$.
\end{ex}

We can then prove the following.

\begin{prop}\label{stronguniqher}
Let $\ca$ be a hereditary abelian category. Then $\D^?(\ca)$ has a strongly unique enhancement, for $?=b,+,-,\emptyset$.
\end{prop}

\begin{proof}
Let $(\cc_1,\fF_1)$ and $(\cc_2,\fF_2)$ be enhancements for $\Da(\ca)$. Since $\ca$ is an abelian category, \autoref{thm:main2} implies that there is an isomorphism $f\colon\cc_1\to\cc_2$ in $\Hqe$. More is true. If $\fQ\colon\Ka(\ca)\to\Da(\ca)$ is the natural quotient functor, let us consider the exact functors
\[
\fG:=\fF_2\comp H^0(f)\comp\fF_1^{-1}\colon\Da(\ca)\to\Da(\ca),\qquad \fH:=\fG\comp\fQ\colon\Ka(\ca)\to\Da(\ca).
\]
By the discussion in \cite[Section 5.1]{CNS} (see, in particular, \cite[Lemma 5.1]{CNS}), $f$ can be chosen in such a way that there exists an isomorphism $\theta\colon\fH\rest{\Va(\ca)}\to\fQ\rest{\Va(\ca)}$, where $\Va(\ca)$ (respectively $\Ba(\ca)$) is the full subcategory of $\Ka(\ca)$ (respectively $\Da(\ca)$) consisting of complexes with $0$ differential.

We need to show that $\fF_2\comp H^0(f)\iso\fF_1$, or, equivalently, that $\fG\iso\id$. Now, the fact that $\ca$ is hereditary implies that the inclusion of $\Ba(\ca)$ in $\Da(\ca)$ is an equivalence (see, for instance, \cite[Section 1.6]{Kr}). Thus it is enough to find an isomorphism $\fG\rest{\Ba(\ca)}\iso\id\rest{\Ba(\ca)}$, and we claim that $\theta$ also defines an isomorphism $\fG\rest{\Ba(\ca)}\to\id\rest{\Ba(\ca)}$ (note that the objects of $\Va(\ca)$ and $\Ba(\ca)$ coincide). So we have to prove that, for any morphism $g\colon A\to B$ in $\Ba(\ca)$, the square
\begin{equation}\label{eq:comm}
\xymatrix{
\fH(A)=\fG(A) \ar[rr]^-{\fG(g)} \ar[d]_-{\theta_A} & & \fG(B)=\fH(B) \ar[d]^-{\theta_B} \\
\fQ(A)=A \ar[rr]_-g & & B=\fQ(B)
}
\end{equation}
commutes in $\Da(\ca)$. As $g$ can be represented by a roof in $\Ka(\ca)$
\[
\xymatrix{
A \ar[r]^-{g_1} & C & B \ar[l]_-{g_2}
}
\]
(where $g_2$ is a quasi-isomorphism), by \cite[Proposition 5.2]{CNS} there exists an isomorphism $\widetilde\theta_C\colon\fH(C)\to\fQ(C)$ such that the diagram
\[
\xymatrix{
\fH(A) \ar[rr]^-{\fH(g_1)} \ar[d]_-{\theta_A} & & \fH(C) \ar[d]^-{\widetilde\theta_C} & & \fH(B) \ar[ll]_-{\fH(g_2)} \ar[d]^-{\theta_B}\\
\fQ(A) \ar[rr]_-{\fQ(g_1)} & & \fQ(C) & & \fQ(B) \ar[ll]^-{\fQ(g_2)}
}
\]
commutes in $\Da(\ca)$. Since $\fQ(g_2)^{-1}\comp\fQ(g_1)=g$ (hence also $\fH(g_2)^{-1}\comp\fH(g_1)=\fG(g)$), this proves that \eqref{eq:comm} commutes.
\end{proof}

The following is then a straightforward application.

\begin{cor}\label{cor:stroncases}
The following triangulated categories have strongly unique enhancement:
\begin{itemize}
\item $\Da(\Mod{R})$ (and $\Da(\fMod{R})$), where $R$ is a left hereditary (and left noetherian) ring,
\item $\Da(\Qcoh(C))$ and $\Da(\Coh(C))$, where $C$ is a smooth curve over a field,
\end{itemize}
for $?=b,+,-,\emptyset$.
\end{cor}

\section{Strong uniqueness for the dual numbers}\label{stronguniq}

By \autoref{DSQp} $\Da$, for $?=sb,b,+,-,\emptyset$, and $\Dmfg$ have strongly unique enhancements if and only if the same is true for $\SQpa$ and $\SQpmfg$. In the following we will always deal with the latter categories. 

We first assume $?=b,+,-,\emptyset$. Since it has a unique enhancement by \autoref{thm:main2}, $\SQpa$ has a strongly unique enhancement if and only if every exact autoequivalence $\fF$ of $\SQpa$ admits a dg lift to an isomorphism $f\in\Hom_{\Hqe}(\cC,\cC)$, where $\cC$ is an arbitrary fixed dg enhancement of $\SQpa$.

We start by proving the following result.

\begin{lem}\label{rest}
By restriction $\fF$ induces an exact autoequivalence of $\SQpbfg$.
\end{lem}

\begin{proof}
From \autoref{indchar} and \autoref{indec2} we immediately deduce that the objects of $\SQpbfg$ are precisely the finite coproducts of the indecomposable objects of $\SQpa$. As $\fF$ is an equivalence, this clearly implies that $\fF(\SQpbfg)\subseteq\SQpbfg$, and so we obtain an exact and fully faithful functor $\fF\rest{\SQpbfg}\colon\SQpbfg\to\SQpbfg$. By applying the same argument to $\fF^{-1}$ we conclude that $\fF\rest{\SQpbfg}$ is also essentially surjective.
\end{proof}

Now, by \cite[Theorem 7.1]{CY}, $\fF\rest{\SQpbfg}$ admits a dg lift, which can be clearly extended to an isomorphism $f\in\Hom_{\Hqe}(\cC,\cC)$. Up to replacing $\fF$ with $\fF\comp H^0(f)^{-1}$, we could conclude that $\SQpa$ has a strongly unique enhancement provided the following question has a positive answer.

\begin{qn}\label{mainqn}
Let $\fF$ be an exact autoequivalence of $\SQpa$, for $?=b,+,-,\emptyset$, such that $\fF\rest{\SQpbfg}\iso\id_{\SQpbfg}$. Does this imply $\fF\iso\id_{\SQpa}$?
\end{qn}

\begin{remark}\label{funiso}
In \autoref{mainqn} we can actually assume $\fF\rest{\SQpbfg}=\id_{\SQpbfg}$. In order to see this, we consider a more general setting, since it will be useful several times in the following. Given a functor $\fF\colon\cc\to\cd$ and a collection of isomorphisms $\nat_X\colon\fF(X)\to\fG(X)$ in $\cd$ for every $X\in\cc$, the map on objects defined by $\fG$ extends uniquely to a functor $\fG\colon\cc\to\cd$ such that $\nat\colon\fF\to\fG$ is an isomorphism of functors. Note that $\fG(f)=\nat_Y\comp\fF(f)\comp\nat_X^{-1}$ for every morphism $f\colon X\to Y$ in $\cc$, and for this reason we will say that $\fG$ is obtained from $\fF$ by conjugation with the isomorphisms $\{\nat_X\}_{X\in\cc}$. In particular, given a functor $\fG'\colon\cc'\to\cd$ (for some full subcategory $\cc'$ of $\cc$) and an isomorphism $\nat'\colon\fF\rest{\cc'}\to\fG'$, we can find a functor $\fG\colon\cc\to\cd$ such that $\fG\rest{\cc'}=\fG'$ and an isomorphism $\nat\colon\fF\to\fG$ simply by choosing, for instance, $\nat_X:=\nat'_X$ for $X\in\cc'$ and $\nat_X:=\id_{\fF(X)}$ for $X\in\cc\setminus\cc'$.
\end{remark}

\subsection{Strong uniqueness in the bounded and bounded below cases}

We are now ready to prove part of \autoref{mainthm} (b) by showing that \autoref{mainqn} has a positive answer for $?=b,+$.

We assume (by \autoref{funiso}) $\fF\rest{\SQpbfg}=\id_{\SQpbfg}$. Notice first that, by \autoref{indchar}, \autoref{indec2} and \autoref{indchar2}, the objects of $\SQpa$ are precisely the coproducts, which exist in $\SQpa$, of objects of $\SQpbfg$. Now we choose, for every $\sq{V}\in\SQpa$, such a representation of $\sq{V}$ as a coproduct of objects of $\SQpbfg$, namely morphisms $\inc_\lambda\colon\sq{U}_\lambda\to\sq{V}$ (where $\lambda\in\Lambda$) of $\SQpa$ with $\sq{U}_\lambda\in\SQpbfg$, which satisfy the universal property of coproduct. As $\fF$ (being an equivalence) preserves coproducts and $\fF(\sq{U}_\lambda)=\sq{U}_\lambda$, there exists a unique isomorphism $\nat_{\sq{V}}\colon\fF(\sq{V})\to\sq{V}$ such that $\inc_\lambda=\nat_{\sq{V}}\comp\fF(\inc_\lambda)$ for every $\lambda\in\Lambda$. Observe also that $\nat_{\sq{V}}=\id_{\sq{V}}$ when $\sq{V}\in\SQpbfg$ (because in that case $\fF(\inc_\lambda)=\inc_\lambda$). Therefore, up to conjugation with the $\nat_{\sq{V}}$, we can assume $\fF(\inc_\lambda)=\inc_\lambda$ for every $\lambda\in\Lambda$ (and for every $\sq{V}\in\SQpa$) and that the condition $\fF\rest{\SQpbfg}=\id_{\SQpbfg}$ still holds. We claim that this is enough to conclude, meaning that then $\fF(f)=f$ for every morphism $f$ of $\SQpa$. If $f\colon\sq{U}\to\sq{V}$ with $\sq{U}\in\SQpbfg$, the compactness\footnote{It is not difficult to show directly that every object of $\SQpbfg$ is compact in $\SQpa$, for $?=b,+$. Alternatively, one can invoke \cite[Corollary 6.17]{R} when $?=b$, and then observe that the case $?=+$ follows from an easy truncation argument.} of $\sq{U}$ implies that there exist morphisms $f_i\colon\sq{U}\to\sq{U}_{\lambda_i}$ (for $i=1,\dots,n$ and $\lambda_i\in\Lambda$) of $\SQpbfg$ such that $f=\sum_{i=1}^n\inc_{\lambda_i}\comp f_i$. As $\fF$ is an additive functor, we deduce that $\fF(f)=f$ in this case. In general, when $f\colon\sq{V}\to\sq{W}$ is an arbitrary morphism of $\SQpa$, by what we have already proved we obtain 
\[
\fF(f)\comp\inc_\lambda=\fF(f)\comp\fF(\inc_\lambda)=\fF(f\comp\inc_\lambda)=f\comp\inc_\lambda
\]
for every $\lambda\in\Lambda$, whence $\fF(f)=f$.

\subsection{The remaining cases of \autoref{mainthm} (b)}

The proof of strong uniqueness for $\SQpsb$ and $\SQpmfg$ uses arguments similar to the one above for $\SQpb$ and $\SQpu$. Thus we just sketch them, emphasizing what needs to be changed.

As for $\SQpsb$, uniqueness of enhancement holds thanks to \cite{CNS3}. Moreover, one can easily show that every exact autoequivalence of $\SQpsb$ restricts to an exact autoequivalence of $\SQpc$. Using the fact that the objects of $\SQpsb$ are precisely the coproducts, which exist in $\SQpsb$, of objects of $\SQpc$, the rest of the proof works as before.

Passing to $\SQpmfg$, again by \autoref{indchar2}, every object $\sq{V}$ of $\SQpmfg$ is a coproduct of objects of $\SQpbfg$. Since $\dim_\kk(V^i)<\infty$ for every $i\in\ZZ$, it should be clear that every such coproduct is also a product. Then one can rather directly conclude using the universal properties of coproduct and product.

\section{Speculations on the unbounded and bounded above cases}\label{open}

We want to conclude the paper with some additional speculations. The aim would be to deal with the remaining two cases: the unbounded and the bounded above cases. Unfortunately at the moment we are unable to decide whether the enhancements for these two categories are strongly unique or not, while they are certainly unique in view of \autoref{thm:main2}. In \autoref{results} we exhibit the partial results we have in this direction, while in \autoref{spec} we outline some open problems. 

\subsection{Partial results }\label{results}

Let us now summarize what we can prove so far. From now on, we assume $?=-,\emptyset$. We set $\Sqpadc:=\Sqpa\cap\Sqd[\ZZ,\ZZ]$ and $\SQpadc:=\SQpa\cap\SQd[\ZZ,\ZZ]$.

\begin{lem}\label{idSQpadc}
We can assume $\fF\rest{\SQpadc}=\id_{\SQpadc}$.
\end{lem}

\begin{proof}
This is similar to the proof in \autoref{stronguniq}, in view of \autoref{funiso}.
\end{proof}

For every $\sq{V}\in\Sq$, setting $\cotr{\sq{V}}:=\bigoplus_{n\le0}V^{\ge n}$, it is easy to see that there is an exact sequence $0\to\cotr{\sq{V}}\mor{\nata_{\sq{V}}}\cotr{\sq{V}}\mor{\natb_{\sq{V}}}\sq{V}\to0$ in $\Sq$, where each component $V^{\ge n}\to\cotr{\sq{V}}$ of $\nata_{\sq{V}}$ is the difference between the natural morphisms $V^{\ge n}\to\cotr{\sq{V}}$ and $V^{\ge n}\to V^{\ge n-1}\to\cotr{\sq{V}}$, while each component $V^{\ge n}\to\sq{V}$ of $\natb_{\sq{V}}$ is just the natural morphism, which we will denote by $\natb_{\sq{V},n}$. By \autoref{triangles} there exists unique $\natc_{\sq{V}}\in\Homep_\Sq(\sq{V},\sh{\cotr{\sq{V}}})$ such that the triangle
\begin{equation}\label{distrunc}
\cotr{\sq{V}}\mor{\nata_{\sq{V}}}\cotr{\sq{V}}\mor{\natb_{\sq{V}}}\sq{V}\mor{\natc_{\sq{V}}}\sh{\cotr{\sq{V}}}
\end{equation}
is distinguished in $\SQ$. For later use we also point out that for every $n\in\ZZ$ there is a natural exact sequence $0\to V^{\ge n}\mor{\natb_{\sq{V},n}}\sq{V}\mor{\natd_{\sq{V},n}}V^{<n}\to0$ in $\Sq$. Hence, similarly as above, there exists unique $\nate_{\sq{V},n}\in\Homep_\Sq(V^{<n},\sh{V^{\ge n}})$ such that the triangle
\begin{equation}\label{distruncn}
V^{\ge n}\mor{\natb_{\sq{V},n}}\sq{V}\mor{\natd_{\sq{V},n}}V^{<n}\mor{\nate_{\sq{V},n}}\sh{V^{\ge n}}
\end{equation}
is distinguished in $\SQ$.

\begin{remark}\label{trdc}
If $\sq{V}\in\Sqpa$, then certainly $V^{\ge n}\in\Sqpu\cap\Sqpa$, so that $V^{\ge n}\in\Sqpad$ by \autoref{indchar} and \autoref{indchar2}. As clearly $V^{\ge n}\in\Sqd[\ZZ_{\ge n},\ZZ]$, we see that $\cotr{\sq{V}}\in\Sqpadc$.
\end{remark}

\begin{prop}\label{mainhyp}
We can assume that $\fF$ acts as the identity on the morphisms with source in $\SQpadc$ and on the morphisms of the triangles \eqref{distrunc} (for every $\sq{V}\in\SQpa$).
\end{prop}

\begin{proof}
By \autoref{idSQpadc} we can assume $\fF\rest{\SQpadc}=\id_{\SQpadc}$. Thus, due to \autoref{trdc}, $\fF(\nata_{\sq{V}})=\nata_{\sq{V}}$ for every $\sq{V}\in\SQpa$. Since $\fF$ is exact, there exists an isomorphism $\nat_{\sq{V}}\colon\fF(\sq{V})\to\sq{V}$ in $\SQpa$ such that the diagram
\[
\xymatrix{
 & & \fF(\sq{V}) \ar[dd]^{\nat_{\sq{V}}} \ar[drr]^{\fF(\natc_{\sq{V}})} \\
\cotr{\sq{V}} \ar[rru]^{\fF(\natb_{\sq{V}})} \ar[rrd]_{\natb_{\sq{V}}} & & & & \sh{\cotr{\sq{V}}} \\
 & & \sq{V} \ar[rru]_{\natc_{\sq{V}}} 
}
\]
commutes for every $\sq{V}\in\SQpa$. We can clearly choose $\nat_{\sq{V}}=\id_{\sq{V}}$ when $\sq{V}\in\SQpadc$. Therefore, up to conjugation with the $\nat_{\sq{V}}$, we can also assume $\fF(\natb_{\sq{V}})=\natb_{\sq{V}}$ and $\fF(\natc_{\sq{V}})=\natc_{\sq{V}}$ for every $\sq{V}\in\SQpa$. Finally, for every morphism $f\colon\sq{W}\to\sq{V}$ of $\SQpa$ with $\sq{W}$ in $\SQpadc$ there exists $g\in\Hom_{\SQpadc}(\sq{W},\cotr{\sq{V}})$ such that $f=\natb_{\sq{V}}\comp g$. Indeed, this follows from the fact that $\sq{W}$ is a coproduct of objects of the form $\sq{S}_{m,n}$ (with $m,n\in\ZZ$ and $m\le n$), and obviously every morphism $\sq{S}_{m,n}\to\sq{V}$ factors through the natural morphism $V^{\ge m}\to\sq{V}$. Hence we conclude that $\fF(f)=f$, as we already know that $\fF(\natb_{\sq{V}})=\natb_{\sq{V}}$ and $\fF(g)=g$.
\end{proof}

From now on we assume that $\fF$ is an exact autoequivalence of $\SQpa$ as in \autoref{mainhyp}. Observe that, in particular, $\fF$ is the identity on objects.

Recall that a morphism $f\colon X\to Y$ in a compactly generated triangulated category (with coproducts) $\ct$ is \emph{phantom} if $f\comp g=0$ for every morphism $g\colon C\to X$ with $C$ compact. It is immediate from the definition that phantom morphisms form an ideal in $\ct$.

For every $\sq{V},\sq{W}\in\Sqp$ we denote by $\Homph_\Sq(\sq{V},\sq{W})$ the subspace of $\Hom_{\SQ}(\sq{V},\sq{W})$ consisting of phantom morphisms.

\begin{remark}\label{phep}
A morphism $f\colon\sq{V}\to\sq{W}$ of $\SQp$ is phantom if and only if $f\comp\natb_{\sq{V},n}=0$ for every $n\in\ZZ$. Indeed, the condition is necessary because $V^{\ge n}\in\Sqpdc$ by \autoref{trdc}. On the other hand, the condition is sufficient because every morphism $\sq{U}\to\sq{V}$ of $\SQ$ with $\sq{U}\in\Sqpc$ factors through $\natb_{\sq{V},n}\colon V^{\ge n}\to\sq{V}$, where $n$ is such that $U^i=0$ for $i<n$. It follows that $f$ is phantom if and only if $f\comp\natb_{\sq{V}}=0$. Thanks to the distinguished triangle \eqref{distrunc} this last condition is satisfied if and only if $f$ factors through $\natc_{\sq{V}}$ (in particular, $\natc_{\sq{V}}$ is phantom). As $\natc_{\sq{V}}\in\Homep_\Sq(\sq{V},\sh{\cotr{\sq{V}}})$, this implies
\[
\Homph_\Sq(\sq{V},\sq{W})\subseteq\Homep_\Sq(\sq{V},\sq{W})
\]
for every $\sq{V},\sq{W}\in\Sqp$. In other words, every phantom morphism of $\SQp$ is of type $\ep$.
\end{remark}

\begin{remark}\label{extraiso}
The assumptions of \autoref{mainhyp} are still satisfied if we change $\fF$ by conjugation with (iso)morphisms of $\SQpa$ of the form $\id_{\sq{V}}+p_{\sq{V}}$ with $p_{\sq{V}}\colon\sq{V}\to\sq{V}$ phantom for every $\sq{V}\in\Sqpa$. This is simply due to the fact that every phantom morphism with source in $\SQpadc$ is trivial and every morphism in a triangle \eqref{distrunc} either has source in $\SQpadc$ (namely, $\nata_{\sq{V}}$ and $\natb_{\sq{V}}$) or is of type $\ep$ (namely, $\natc_{\sq{V}}$).
\end{remark}

\begin{prop}\label{diffph}
For every morphism $f\colon\sq{V}\to\sq{W}$ of $\SQpa$ we have $\fF(f)-f\in\Homph_\Sq(\sq{V},\sq{W})$.
\end{prop}

\begin{proof}
By assumption, and recalling that $\cotr{\sq{V}}\in\Sqpadc$ by \autoref{trdc}, we have
\[
  \fF(f)\comp\natb_{\sq{V}}=\fF(f)\comp\fF(\natb_{\sq{V}})=\fF(f\comp\natb_{\sq{V}})=f\comp\natb_{\sq{V}}.
\]
Thus we conclude by \autoref{phep}.
\end{proof}

\begin{cor}\label{injtar}
If $f\colon\sq{V}\to\sq{W}$ is a morphism of $\SQpa$ with $\sq{W}$ injective in $\Sq$, then $\fF(f)=f$.
\end{cor}

\begin{proof}
By \autoref{phep} and \autoref{Ext1} we have
\[
\Homph_\Sq(\sq{V},\sq{W})\subseteq\Homep_\Sq(\sq{V},\sq{W})\iso\Ext^1_\Sq(\sq{V},\sh[-1]{\sq{W}})=0
\]
(since $\sh[-1]{\sq{W}}$ is injective). The conclusion then follows from \autoref{diffph}.
\end{proof}

\begin{lem}\label{idph}
If $f$ is a phantom morphism of $\SQpa$, then $\fF(f)=f$.
\end{lem}

\begin{proof}
If $f\in\Homph_{\Sqpa}(\sq{V},\sq{W})$, then by \autoref{phep} there exists $g\in\Hom_{\SQpa}(\sh{\cotr{\sq{V}}},\sq{W})$ such that $f=g\comp\natc_{\sq{V}}$. By assumption $\fF(g)=g$ (recall that $\sh{\cotr{\sq{V}}}\in\Sqpadc$ by \autoref{trdc}) and $\fF(\natc_{\sq{V}})=\natc_{\sq{V}}$, whence $\fF(f)=f$.
\end{proof}

\begin{lem}\label{idkdc}
If $f\colon\sq{S}_{-\infty,n}\to\sq{V}$ is a morphism of $\SQpa$ with $\sq{V}\in\SQpadc$ and $n\in\ZZ$, then $\fF(f)=f$.
\end{lem}

\begin{proof}
First note that $\Hom_\Sq(\sq{S}_{-\infty,n},\sq{V})=0$, since $\compat{\sq{V}}=0$ (see \autoref{Homindec}). This means that $f$ is of type $\ep$, and we claim that $\natd_{\sq{V},n}\comp f\colon\sq{S}_{-\infty,n}\to V^{<n}$ is phantom. In order to see this, by \autoref{phep} it is enough to show that $\natd_{\sq{V},n}\comp f\comp\natb_{\sq{S}_{-\infty,n},m}\colon\sq{S}_{m,n}\to V^{<n}$ is $0$ for every $m\le n$. In fact, this last morphism is of type $\ep$ and it is very easy to check that $\Homep_\Sq(\sq{S}_{m,n},\sq{W})=0$ if $\sq{W}\in\Sq$ is such that $W^i=0$ for $i\ge n$. Thus $\fF(\natd_{\sq{V},n}\comp f)=\natd_{\sq{V},n}\comp f$ by \autoref{idph} and, as also $\fF(\natd_{\sq{V},n})=\natd_{\sq{V},n}$ by assumption, we obtain
\[
\natd_{\sq{V},n}\comp\fF(f)=\fF(\natd_{\sq{V},n})\comp\fF(f)=\fF(\natd_{\sq{V},n}\comp f)=\natd_{\sq{V},n}\comp f.
\]
This proves that $g:=\fF(f)-f\colon\sq{S}_{-\infty,n}\to\sq{V}$ satisfies $\natd_{\sq{V},n}\comp g=0$. Then from the distinguished triangle \eqref{distruncn} we deduce that there exists $h\in\Hom_{\SQpa}(\sq{S}_{-\infty,n},V^{\ge n})$ such that $g=\natb_{\sq{V},n}\comp h$. Now, it is easy to see that there is a direct summand $\sq{\tilde{V}}\in\SQpc$ of $\sq{V}$ such that $h$ factors through the natural inclusion $\tilde{V}^{\ge n}\subseteq V^{\ge n}$, which implies that $g$ is the composition of $\tilde{g}\in\Hom_{\SQpa}(\sq{S}_{-\infty,n},\sq{\tilde{V}})$ and of the inclusion $\sq{\tilde{V}}\subseteq\sq{V}$. On the other hand, $g$ (hence also $\tilde{g}$) is phantom by \autoref{diffph}, whence for $m\le n$ we have $\tilde{g}\comp\natb_{\sq{S}_{-\infty,n},m}=0\colon\sq{S}_{m,n}\to\sq{\tilde{V}}$. So by \eqref{distruncn} $\tilde{g}$ factors through $\natd_{\sq{S}_{-\infty,n},m}\colon\sq{S}_{-\infty,n}\to\sq{S}_{-\infty,m-1}$. Choosing $m$ such that $\tilde{V}^i=0$ for $i<m$, we get $\tilde{g}=0$ (because $\Hom_{\SQpa}(\sq{S}_{-\infty,m-1},\sq{\tilde{V}})=0$), hence $g=0$. 
\end{proof}

\begin{prop}\label{idep}
If $f$ is a morphism of type $\ep$ of $\SQpa$, then $\fF(f)=f$.
\end{prop}

\begin{proof}
Let $f\in\Homep_{\SQpa}(\sq{U},\sq{V})$. Since $\Sq$ has enough injectives, there exists an exact sequence
\begin{equation}\label{injres}
0\to\sq{U}\mor{l}\sq{W}\to\sq{\quot{W}}\to0
\end{equation}
in $\Sq$ with $\sq{W}$ injective, and we claim that we can suppose $\sq{W}$ (hence also $\sq{\quot{W}}=\cok(l)$) to be in $\Sqpa$. First observe that $\sq{W}\in\Sqd[\{-\infty\},\ZZ\cup\{\infty\}]$ by \autoref{injdec}. As every morphism from $\sq{U}$ to an object of $\Sqd[\{-\infty\},\{\infty\}]$ is trivial (see \autoref{DSQp} and \autoref{Sqachar}), in any case we can assume $\sq{W}\in\Sqd[\{-\infty\},\ZZ]\subseteq\Sqp$. Moreover, if $\sq{U}\in\Sqpm$, then, by \autoref{Sqpachar}, there exists $n\in\ZZ$ such that $U^i=0$ for every $i>n$. This easily implies that every copy in $\sq{W}$ of $\sq{S}_{-\infty,i}$ with $i>n$ can be removed. Thus we can assume $\sq{W}\in\Sqd[\{-\infty\},\ZZ_{\le n}]\subseteq\Sqpm$.

Denoting by $e\in\Homep_\Sq(\sh[-1]{\sq{\quot{W}}},\sq{U})$ the morphism such that $\sh{e}$ corresponds (by \autoref{Ext1}) to the isomorphism class of \eqref{injres} in $\Ext^1_\Sq(\sq{\quot{W}},\sq{U})$, by \autoref{triangles} there is a distinguished triangle $\sh[-1]{\sq{\quot{W}}}\mor{e}\sq{U}\mor{l}\sq{W}$ in $\SQpa$. Then from $f\comp e=0$ (which holds because both $f$ and $e$ are of type $\ep$) we deduce the existence of $g\in\Hom_{\SQpa}(\sq{W},\sq{V})$ such that $f=g\comp l$, and we can assume $g$ to be of type $\ep$ (since $l$ is of type $1$). Similarly, using the distinguished triangle \eqref{distrunc} (and the fact that $\natb_{\sq{V}}$ is of type $1$ and $\natc_{\sq{V}}$ of type $\ep$), from $\natc_{\sq{V}}\comp g=0$ we obtain the existence of $h\in\Hom_{\SQpa}(\sq{W},\cotr{\sq{V}})$ such that $g=\natb_{\sq{V}}\comp h$ (hence $f=\natb_{\sq{V}}\comp h\comp l$), and again we can assume $h$ to be of type $\ep$. Now, $\fF(\natb_{\sq{V}})=\natb_{\sq{V}}$ by assumption and $\fF(l)=l$ by \autoref{injtar}. Therefore, in order to conclude that $\fF(f)=f$, it is enough to prove that $\fF(h)=h$. By what we have already observed, $\sq{W}\iso\bigoplus_{\lambda\in\Lambda}\sq{W}_\lambda$, where each $\sq{W}_\lambda$ is of the form $\sq{S}_{-\infty,n}$ for some $n\in\ZZ$. Denoting by $\inc_\lambda\colon\sq{W}_\lambda\to\sq{W}$ the natural morphism, $\fF(h)=h$ follows from the fact that for every $\lambda\in\Lambda$ we have
\[
\fF(h)\comp\inc_\lambda=\fF(h)\comp\fF(\inc_\lambda)=\fF(h\comp\inc_\lambda)=h\comp\inc_\lambda,
\]
where the first equality holds because $\fF(h)$ and $\fF(\inc_\lambda)-\inc_\lambda$ are of type $\epsilon$ (by \autoref{diffph}), whereas the last one is due to \autoref{trdc} and \autoref{idkdc}.
\end{proof}

Now we choose, for every $\sq{V}\in\Sqpa$, a subobject $\ninj{\sq{V}}\subseteq\sq{V}$ such that
\begin{equation}\label{decomp}
\sq{V}=\inj{\sq{V}}\oplus\ninj{\sq{V}}
\end{equation}
(this is possible thanks to \autoref{injpart}).

\begin{lem}\label{extrahyp}
We can assume that, beyond being as in \autoref{mainhyp}, $\fF$ acts as the identity on the inclusions of $\inj{\sq{V}}$ and $\ninj{\sq{V}}$ in $\sq{V}$, for every $\sq{V}\in\SQpa$.
\end{lem}

\begin{proof}
More generally, if $\sq{V}$ is a coproduct in $\Sqpa$, given by morphisms $\inc_\lambda\colon\sq{V}_\lambda\to\sq{V}$ (where $\lambda\in\Lambda$), let $\nat_{\sq{V}}\colon\fF(\sq{V})\to\sq{V}$ be the unique isomorphism of $\SQpa$ such that $\inc_\lambda=\nat_{\sq{V}}\comp\fF(\inc_\lambda)$ for every $\lambda\in\Lambda$ (see the proof of \autoref{idSQpadc}). Since each $\inc_\lambda$ is of type $1$ and $\fF(\inc_\lambda)-\inc_\lambda$ is phantom (by \autoref{diffph}), it is clear that $\nat_{\sq{V}}=\id_{\sq{V}}+p_{\sq{V}}$ with $p_{\sq{V}}\colon\sq{V}\to\sq{V}$ of type $\ep$ and such that $p_{\sq{V}}\comp\inc_\lambda$ is phantom for every $\lambda\in\Lambda$. Actually this last condition easily implies (using the fact that every morphism $\sq{U}\to\sq{V}$ with $\sq{U}\in\SQpc$ factors through $\bigoplus_{\lambda\in\Lambda'}\sq{V}_\lambda$ for some finite subset $\Lambda'$ of $\Lambda$) that $p_{\sq{V}}$ is phantom, as well. If we consider, for every $\sq{V}\in\SQpa$, the decomposition as a coproduct given by \eqref{decomp}, we conclude (remembering \autoref{extraiso}) that $\fF$ has the required properties, up to conjugation with the $\nat_{\sq{V}}$. 
\end{proof}

\begin{prop}\label{idninj}
Let $\fF$ be as in \autoref{extrahyp}. Then for every morphism of $\SQpa$
\[
f=
\begin{pmatrix}
f_{i,i} & f_{n,i} \\
f_{i,n} & f_{n,n}
\end{pmatrix}
\colon\sq{V}=\inj{\sq{V}}\oplus\ninj{\sq{V}}\to\sq{W}=\inj{\sq{W}}\oplus\ninj{\sq{W}}
\]
(with respect to the fixed decompositions given by \eqref{decomp}) we have $\fF(f)=f$ if and only if $\fF(f_{n,n})=f_{n,n}$. In particular, $\fF(f)=f$ if $\sq{V}\in\SQpad$.
\end{prop}

\begin{proof}
Clearly $\fF(f)=f$ if and only if $\fF(f_{l,m})=f_{l,m}$ for $l,m\in\{i,n\}$. Then the first statement follows from the fact that $\fF(f_{i,i})=f_{i,i}$, $\fF(f_{n,i})=f_{n,i}$ (by \autoref{injtar}) and $\fF(f_{i,n})=f_{i,n}$ (by \autoref{idep}, since $f_{i,n}$ is of type $\ep$ as an easy consequence of \autoref{Homindec}). As for the last statement, it is enough to observe that, if $\sq{V}\in\SQpad$, then $\ninj{\sq{V}}\in\SQpadc$; hence $\fF(f_{n,n})=f_{n,n}$ in this case.
\end{proof}

\subsection{Open problems}\label{spec}

In the setting of \autoref{results}, $\fF$ is the identity on objects and it is determined by its action on morphisms of type $1$ (by \autoref{idep}). Moreover, $\fD(f):=\fF(f)-f\in\Homph_\Sq(\sq{V},\sq{W})$ for every $f\in\Hom_{\Sqpa}(\sq{V},\sq{W})$ (by \autoref{diffph}). It is immediate to see that $\fF$ is a ($\kk$-linear) functor if and only if $\fD$ is a \emph{derivation}, meaning that $\fD$ is given by $\kk$-linear maps such that $\fD(g\comp f)=\fD(g)\comp f+g\comp\fD(f)$ whenever $f$ and $g$ are composable morphisms of $\Sqpa$. On the other hand, $\fF\iso\id_{\SQpa}$ if and only if $\fD$ is an \emph{inner derivation}, meaning that there exists $\theta_{\sq{V}}\in\Homep_\Sq(\sq{V},\sq{V})$ for every $\sq{V}\in\Sqpa$ such that $\fD(f)=f\comp\theta_{\sq{V}}-\theta_{\sq{W}}\comp f$ for every $f\in\Hom_{\Sqpa}(\sq{V},\sq{W})$.

\begin{qn}\label{noninner}
Is every derivation (in the above sense) inner?
\end{qn}

Clearly if the answer to \autoref{noninner} is yes, then the answer to \autoref{mainqn} is yes, too. Thus we would be able to prove strong uniqueness for the enhancements for $\Da(\Mod{\kep})$, for $?=\emptyset,-$. However, if the answer to \autoref{noninner} is no, then a priori we can only conclude that there exists a ($\kk$-linear, but not necessarily exact) autoequivalence $\fF$ of $\SQpa$ which is not isomorphic to the identity. 

In order to be able to answer \autoref{mainqn} or \autoref{noninner} it could be useful to find some sort of classification of the (non completely decomposable) objects of $\Sqpa$ (for $?=\emptyset,-$).

\bigskip

{\small\noindent{\bf Acknowledgements.} Part of this work was carried out while the third author was visiting the Institut des Hautes \'Etudes Scientifiques (Paris), the SLMAth Institute (Berkeley) and the Laboratoire de Math\'ematiques d'Orsay (Universit\'e Paris-Saclay) whose warm hospitality is gratefully acknowledged.}


\end{document}